\documentclass[12pt]{amsart}
\usepackage{amsmath,amssymb}
\usepackage{bm}
\usepackage{xcolor}
\usepackage{mathrsfs}
\usepackage{verbatim}
\usepackage{float}
\usepackage{setspace}
\usepackage{amsthm}
\usepackage{geometry}
\usepackage{color}
\usepackage{cite}
\allowdisplaybreaks[4]
\newtheorem{theorem}{Theorem}[section]
\newtheorem{lemma}[theorem]{Lemma}
\usepackage{setspace}

\usepackage{appendix}
\usepackage{chemarrow}
\usepackage{extarrows}

\newcommand{\mn}{\bm{n}}

\newcommand{\mt}{\mathcal{T}_h}

\newcommand{\ms}{\mathbb{S}}

\newcommand{\mb}{\mathbb{R}}
\newcommand{\be}{\begin{eqnarray}}
	\newcommand{\ee}{\end{eqnarray}}
\newcommand{\ea}{\end{array}}
\newcommand{\ba}{\begin{array}}
\newcommand{\pb}{\begin{frame}}
	\newcommand{\pe}{\end{frame}}

\newcommand{\rk}{{\rm{ker}}} 

\newcommand{\bsi}{\bm{\sigma}}
\newcommand{\bta}{\bm{\tau}}

\newcommand{\ben}{\begin{equation*}}
	\newcommand{\een}{\end{equation*}}

\newtheorem{The}{Theorem}[section]

\newtheorem{lm}{Lemma}[section]

\newtheorem{re}{Remark}[section]

\numberwithin{equation}{section}
\numberwithin{table}{section}
\numberwithin{figure}{section}

\theoremstyle{remark}
\newtheorem{remark}[theorem]{Remark}

\DeclareMathOperator{\ggrad}{grad}
\DeclareMathOperator{\ccurl}{curl}
\DeclareMathOperator{\ddiv}{div}
\DeclareMathOperator{\ssym}{sym}
\DeclareMathOperator{\ddev}{dev}  
\DeclareMathOperator{\rrot}{rot} 
\DeclareMathOperator{\mmod}{mod} 
\newcommand{\rd}{\ddiv}
\newcommand{\bd}{\ddiv}
\newcommand{\bc}{\ccurl}
  \makeatletter
  \@addtoreset{equation}{section}
  \makeatother

\title[New Conforming Finite element divdiv Complexes in Three Dimensions]{New Conforming Finite element  divdiv Complexes in Three Dimensions}

\author {Jun Hu}
\address{LMAM and School of Mathematical Sciences, Peking University,
  Beijing 100871, P. R. China. }
\email{hujun@math.pku.edu.cn}

\author {Yizhou Liang}
\address{LMAM and School of Mathematical Sciences, Peking University,
  Beijing 100871, P. R. China. }
\email{lyz2015@pku.edu.cn}

\author {Rui Ma}
\address{School of Mathematics and Statistics, Beijing Institute of Technology, 
Beijing 100081, P.R. China. }
\email{rui.ma@bit.edu.cn }

 \author {Min Zhang}
 \address{Computational Science Research Center, Beijing {\rm{100193}}, P. R. China.}
\address{School of Mathematical Sciences, Peking University,
	Beijing {\rm100871}, P. R. China.}
\email{zhangminyy@csrc.ac.cn }

\begin{document}
\maketitle
\begin{abstract}
In this paper, the first family of conforming finite element divdiv complexes on cuboid grids in three dimensions is constructed. Besides, a new family of conforming finite element divdiv complexes with enhanced smoothness on tetrahedral grids is presented. These complexes are exact in the sense that the range of each discrete map is the kernel space of the succeeding one.
\end{abstract}

\section{Introduction}
This paper considers the construction of conforming finite element sub-complexes of the following divdiv complex on $\Omega\subseteq \mathbb{R}^{3}$, see \cite{MR4113080, 2020Complexes}
\begin{equation}\label{divdivComp}
\begin{aligned}
RT\stackrel{\subseteq}{\rightarrow}H^1(\Omega;\mathbb{R}^3) \stackrel{\operatorname{dev}\ggrad}{\longrightarrow}H(\operatorname{sym}\ccurl,\Omega;&\mathbb{T}) \stackrel{\operatorname{sym}\ccurl}{\longrightarrow}H(\operatorname{div}\ddiv,\Omega;\mathbb{S})\\ &\stackrel{\ddiv\ddiv}{\longrightarrow}L^2(\Omega;\mathbb{R})\stackrel{}{\rightarrow}0,
\end{aligned}
\end{equation}
where $RT:=\{a\mathbf{x}+\bm{b}:\, a\in\mb, \bm{b}\in \mb^{3}\}$, the spaces $H^1(\Omega;\mathbb{R}^3)$ and $L^2(\Omega;\mathbb{R})$ are the vector-valued and scalar-valued Sobolev spaces, respectively, and 
\begin{equation*}
	\begin{aligned}
		H(\ssym\ccurl, \Omega ; \mathbb{T}):=&\left\{\bsi \in L^{2}(\Omega ; \mathbb{T}): \ssym\ccurl \bsi\in L^{2}(\Omega ; \mathbb{S})\right\},\\
		H(\ddiv\ddiv, \Omega ; \mathbb{S}):=&\left\{\bsi \in L^{2}(\Omega ; \mathbb{S}): \ddiv\ddiv \bsi \in L^{2}(\Omega ; \mathbb{R})\right\}.
	\end{aligned}
\end{equation*} 
Here $\mathbb{T}$ and $\ms$ denote the spaces of traceless and symmetric matrices in three dimensions, respectively, the operator $\ssym\ccurl$ is the symmetric part of the row-wise $\ccurl$ operator, and the inner divergence operator of $\ddiv\ddiv$ is applied by row resulting a column vector for which the outer divergence operator is applied.
Such a complex is exact provided that the domain $\Omega$ is contractible and Lipschitz \cite{MR4113080,2020Complexes}. Particularly,  conforming finite element spaces $\Sigma_{h}\subseteq H(\ddiv\ddiv,\Omega;\mathbb{S})$, $V_{h} \subseteq H^1(\Omega;\mathbb{R}^3)$, $\mathcal{U}_{h}\subseteq H(\operatorname{sym}\ccurl,\Omega;\mathbb{T})$, and $Q_{h}\subseteq L^{2}(\Omega;\mathbb{R})$ are constructed, such that the following discrete divdiv complex 
\begin{equation}\label{eq:DisComp}
R T \stackrel{\subseteq}{\longrightarrow} V_{h} \stackrel{\operatorname{dev} \operatorname{grad}}{\longrightarrow} \mathcal{U}_{h} \stackrel{\text { sym curl }}{\longrightarrow} \Sigma_{h} \stackrel{\text { div div }}{\longrightarrow}Q_{h} \rightarrow 0
\end{equation}
is exact.

One of the applications of the finite element divdiv complex of \eqref{eq:DisComp} is to solve the fourth order problems \cite{2020arXiv200501271C, 2020arXiv200712399C, 2020arXiv201002638H, MR4113080}. 
Besides, the construction of the finite element divdiv complex of \eqref{eq:DisComp} is closely related to the mixed formulation of the so-called linearized Einstein-Bianchi system \cite{MR3450102}. The first finite element sub-complex of the gradgrad complex (the dual complex of the divdiv complex) was constructed in \cite{2021Conforming}, and the finite element spaces were employed to solve the linearized Einstein-Bianchi system within the mixed form.
Recently, two finite element divdiv complexes of \eqref{eq:DisComp} on tetrahedral grids were constructed in \cite{2021arXiv210300088H,2020arXiv200712399C}, and the associated finite element spaces can be used to discretize the linearized Einstein-Bianchi system within the dual formulation introduced in \cite{2021arXiv210300088H}.
Both used the $H(\ddiv\ddiv,\Omega;\mathbb{S})$ conforming finite element on tetrahedral grids from \cite{2020arXiv200712399C}.
However, it is arduous to compute the basis functions of the finite element spaces in \cite{2021arXiv210300088H,2020arXiv200712399C}, since the degrees of freedom(DOFs) of those finite elements are in some sense complicated. More precisely, the design of the $H(\ddiv\ddiv,\Omega;\ms)$ conforming finite element in \cite{2020arXiv200712399C} was based on Green's identity \cite[Lemma 4.1]{2020arXiv200712399C}
\be
\label{chenhuang}
\begin{split}
(\ddiv\ddiv\bsi,q)_K=&(\bsi,\nabla^2 q)_K-\sum_{f\in\mathscr{F}(K)}\sum_{e\in\mathscr{E}(f)}(\mn_{f,e}^\intercal\bsi\mn_{f}, q)_{f}\\
&+\sum_{f\in\mathscr{F}(K)}[(\mn_{f}^\intercal \bsi\mn_f,\frac{\partial q}{\partial \mn_f})_f-(2\ddiv_f(\bsi\mn_f)+\frac{\partial(\mn_{f}^{\intercal}\bsi\mn_f)}{\partial\mn_f}, q)_f]
\end{split}
\ee
on tetrahedron $K$. Here $\bsi$ is a symmetric matrix valued function, $q$ is a scalar function, $\mathscr{F}(K)$ is the set of all faces of the tetrahedron $K$, $\mathscr{E}(f)$ is the set of all edges of face $f$, $\mn_f$ is the unit normal vector of $f$, $\mn_{f,e}$ denotes the unit normal vector of $e$ that parallels to $f$, and $\ddiv_f$ is the surficial divergence. 
Based on \eqref{chenhuang}, besides the continuity of $\mn_{f,e}^\intercal\bsi\mn_{f}$ across each edge $e$ and $\mn_{f}^\intercal \bsi\mn_f$ across each interface $f$, the continuity of $2\ddiv_f(\bsi\mn_f)+\frac{\partial(\mn_{f}^{\intercal}\bsi\mn_f)}{\partial\mn_f}$ across each interface $f$ is imposed on the functions of the $H(\ddiv\ddiv,\Omega;\ms)$ conforming finite element space in \cite{2020arXiv200712399C}. As a result, the last continuity brings a set of DOFs which combine two terms $2\ddiv_f(\bsi\mn_f)$ and $\frac{\partial(\mn_{f}^{\intercal}\bsi\mn_f)}{\partial\mn_f}$ on the boundary of $K$. Such a combination makes it difficult to figure out the explicit basis functions of the corresponding $H(\ddiv\ddiv,\Omega;\ms)$ conforming finite element space. In addition, such a combination is inherited by both the $H^1(\Omega;\mathbb{R}^3)$ and $H(\ssym\ccurl,\Omega;\mathbb{T})$ conforming finite element spaces in the finite element complex.
In a recent work \cite{2020arXiv201002638H}, new $H(\ddiv\ddiv,\Omega;\ms)$ conforming finite element spaces with explicit basis functions were constructed on both triangular and tetrahedral grids in a unified way. Instead of \eqref{chenhuang}, the design of \cite{2020arXiv201002638H} was based on the following identity
\ben
\begin{split}
(\ddiv\ddiv\bsi,q)_K=(\bsi,\nabla^2 q)_K-\sum_{f\in\mathscr{F}(K)}(\bsi\mn_f,\nabla q)_{f}+\sum_{f\in\mathscr{F}(K)}(\mn_f^{\intercal}\ddiv\bsi, q)_f.
\end{split}
\een
The continuity of $\bsi\mn_f$ and $\mn_f^\intercal\ddiv\bsi$ across each interface $f$ was imposed on the functions of the $H(\ddiv\ddiv,\Omega;\ms)$ conforming finite element space in \cite{2020arXiv201002638H}. Consequently, the element follows by strengthening the continuity of the normal component of the divergences of the functions of the $H(\ddiv,\Omega;\ms)$ conforming finite elements proposed in \cite{MR3529252, MR3352360, MR3301063}. Here $H(\ddiv,\Omega;\ms)$ is a space of square-integrable tensors with square-integrable divergence, taking values in the space $\ms$ of symmetric matrices. 
As it can be seen from \cite{2020arXiv201002638H} that because of a crucial structure of the $H(\ddiv\ddiv,\Omega;\ms)$ conforming finite element space, inherited from that of the $H(\ddiv,\Omega;\ms)$ finite element space, the basis functions can be easily written out. In addition, the $H(\ddiv\ddiv,\Omega;\ms)$ conforming finite element space in \cite{2020arXiv201002638H} is a little bit smoother than that in \cite{2020arXiv200712399C}. Recently, the construction in \cite{2020arXiv201002638H} has been generalized to arbitrary dimension \cite{2021arbitrary}.

This paper is devoted to construct conforming finite element complexes of \eqref{eq:DisComp} on both cuboid and tetrahedral grids in three dimensions. The key part is to characterize the continuity of the finite element spaces appearing in the discrete complexes of \eqref{eq:DisComp}. 
Motivated by the sufficient continuity condition proposed in \cite{2020arXiv201002638H}, which is actually a characterization of the continuity of $H(\ddiv\ddiv,\Omega;\ms)\cap H(\ddiv,\Omega;\ms)$, a new $H(\ddiv\ddiv,\Omega;\ms)$ conforming finite element space $\Sigma_{h}$ is constructed in this paper.
The enhanced smoothness of $\Sigma_{h}$ brings corresponding enhancement of regularity for the finite element subspaces $\mathcal{U}_{h}\subseteq H(\ssym\ccurl,\Omega;\mathbb{T})$ and $V_{h}\subseteq H^1(\Omega;\mathbb{R}^3)$. 
In fact, a global regularity $\mathcal{U}_{h}\subseteq H^{*}(\ssym\ccurl,\Omega;\mathbb{T})$ is imposed for the $H(\ssym\ccurl,\Omega;\mathbb{T})$ conforming finite element space, where $H^{*}(\ssym\ccurl,\Omega;\mathbb{T})$ is a subspace of $H(\ssym\ccurl,\Omega;\mathbb{T})$ such that the column-wise divergences of matrix valued functions belong to the space $H(\ccurl,\Omega;\mathbb{R}^3)$.
Consequently, a global regularity $V_{h}\subseteq H^1(\ddiv,\Omega;\mathbb{R}^3)$ is required for the $H^{1}(\Omega;\mathbb{R}^{3})$ conforming finite element space, where $H^1(\ddiv,\Omega;\mathbb{R}^3)$ is a subspace of $H^{1}(\Omega;\mathbb{R}^{3})$ such that the divergences of vector valued functions are equally in the space $H^1(\Omega;\mathbb{R})$.
As a consequence, the conforming finite element divdiv complexes of \eqref{eq:DisComp} on cuboid grids and tetrahedral grids in three dimensions constructed in this paper are sub-complexes of the divdiv complex with the following enhanced smoothness:
\begin{equation}\label{divdivComp:higher:regularity}
\begin{aligned}
RT\stackrel{\subseteq}{\rightarrow}H^1(\ddiv,\Omega;&\mathbb{R}^3) \stackrel{\operatorname{dev}\ggrad}{\longrightarrow}H^{\ast}(\operatorname{sym}\ccurl,\Omega;\mathbb{T}) \\
&\stackrel{\operatorname{sym}\ccurl}{\longrightarrow}H(\operatorname{div}\ddiv,\Omega;\mathbb{S})\cap H(\ddiv,\Omega;\ms)\stackrel{\ddiv\ddiv}{\longrightarrow}L^2(\Omega;\mathbb{R})\stackrel{}{\rightarrow}0.
\end{aligned}
\end{equation}

On cuboid grids, the spaces $Q_{h}$, $\Sigma_{h}$, $\mathcal{U}_{h}$, and $V_{h}$ are noted as $Q_{k-2,\Box}$, $\Sigma_{k,\Box}$, $\mathcal{U}_{k,\Box}$, and $V_{k,\Box}$ with $k\geq 3$, respectively, where $Q_{k-2,\Box}$ is the space of discontinuous piecewise polynomials of degree no more than $k-2$ for each variable. For the spaces $\Sigma_{k,\Box}$, $\mathcal{U}_{k,\Box}$, and $V_{k,\Box}$, the associated shape function spaces will be determined by the polynomial de Rham complex and the polynomial divdiv complex, and the corresponding DOFs are identified by imposing the associated continuity requirements sufficiently. On tetrahedral grids, the spaces $Q_{h}$, $\Sigma_{h}$, $\mathcal{U}_{h}$, and $V_{h}$ are noted as $Q_{k-2,\triangle}$, $\Sigma_{k,\triangle}$, $\mathcal{U}_{k+1,\triangle}$, and $V_{k+2,\triangle}$ with $k\geq 4$, respectively, where $Q_{k-2,\triangle}$ is the space of discontinuous piecewise polynomials of degree no more than $k-2$, the space $\Sigma_{k,\triangle}$ is a modification of the $H(\ddiv\ddiv,\Omega;\ms)$ conforming finite element in \cite{2020arXiv201002638H}, the spaces $\mathcal{U}_{k+1,\triangle}$ and $V_{k+2,\triangle}$ are newly proposed herein. 

The rest of the paper is organized as follows. Section 2 introduces the notation. Section 3 constructs the finite element divdiv complex on cuboid grids. Section 4 constructs the finite element divdiv complex on tetrahedral grids.  The exactness of the discrete complexes will be proved in Section 5.

\section{Preliminaries}
Let $\Omega$ be a contractible domain with Lipschitz boundary $\partial \Omega$ of $\mathbb{R}^{3}$.
\ Let $\{\mathcal{T}_{h}\}$ denote a family of shape regular cuboid or tetrahedral grids on $\Omega$ with the mesh size $h$. \ Let $\mathscr{V}$, $\mathscr{E}$, and $\mathscr{F}$ denote the set of vertices,\ edges, and faces,\ respectively.
\ Given element $K\in \mt$, let $\mathscr{V}(K)$, $\mathscr{E}(K)$, and $\mathscr{F}(K)$ denote the sets of vertices, edges, and faces of $K$, respectively. Given an edge $e\in \mathscr{E}(K)$, the unit tangential vector $\bm{t}_{e}$, as well as two unit normal vectors $\bm{n}^{\pm}_{e}$ are fixed.  For a facet $f\in \mathscr{F}(K)$, the unit normal vector ${\bm{n}_{f}}$ as well as two unit tangential vectors $\bm{t}^{\pm}_{f}$ are fixed. The jump of $u$ across an interface $f\in {\mathscr{F}}$ shared by neighboring elements $K^{\pm}$ is defined by
$$\left [u\right]_f:=u|_{K^{+}}-u|_{K^{-}}.$$
When it comes to any boundary face $f\subseteq\partial \Omega$, the jump $[\cdot]_f$ reduces to the trace. Let $\nabla_{pw}^{k}$ denote the elementwise derivative of order $k$ for piecewise functions.

Denote by $\mathbb{M}$ the space of\ $3\times 3$\ real matrices,\ and let $\mathbb{S}$, $\mathbb{K}$, and $\mathbb{T}$ be the subspaces of symmetric, skew-symmetric, and traceless matrices,\ respectively. Let $\bm{I}$ denote the $3\times3$ identity, and $\bm{I}_1:=(1, 0, 0)^{\intercal}$, $\bm{I}_2:=(0, 1, 0)^{\intercal}$, as well as $\bm{I}_3:=(0, 0, 1)^{\intercal}$. For matrix valued function $\bsi:\Omega \rightarrow \mathbb{M}$, define the symmetric part and traceless part as follows, respectively,
\begin{equation*}
	\ssym \bsi: = \frac{1}{2}(\bsi + \bsi^{\intercal}),\quad \ddev \bsi: = \bsi - \frac{1}{3}\operatorname{tr}(\bsi)\bm{I}.
\end{equation*}

 For simplicity of presentation, for $\mathbf{x}=(x, y, z)^{\intercal}\in \Omega$,\ let $x_{i}$ be $x, y, z$ when $i$ takes $1,2$ and $3$, respectively, and $\partial/\partial {x_{i}}$ be abbreviated as $\partial_{i}$. 
For a vector field $\bm{v}=(v_1,v_2,v_3)^{\intercal}$, the gradient applies by row to produce a matrix valued function, namely $\ggrad \bm{v}:=\nabla \bm{v}:=(\partial_{j}v_{i})_{3\times 3}$. Define the symmetric gradient by $\varepsilon(\bm{v}):=\ssym (\nabla \bm{v})$ and define a skew symmetric matrix
\begin{equation*}
	\operatorname{mspn} \bm{v}:=\left(\begin{array}{ccc}
		0 & -v_{3} & v_{2} \\
		v_{3} & 0 & -v_{1} \\
		-v_{2} & v_{1} & 0
	\end{array}\right).
\end{equation*}

Given a facet $f\in \mathscr{F}$ with the unit normal vector $\mn_f$, for a vector $\bm{v}\in \mathbb{R}^3$ as well as a scalar function $q$, define
\ben
\Pi_f\bm{v}:=(\mn_f\times \bm{v})\times \mn_f, \quad \nabla_f q:=\Pi_f \nabla q.
\een
Define the surface symmetric gradient $\varepsilon_f$ by 
\ben
\varepsilon_f(\bm{v}):=\ssym (\nabla_f(\Pi_f\bm{v})),
\een 
and the surface rot  operator is defined by 
\ben
\rrot_f\bm{v}:=(\mn_f\times \nabla)\cdot \bm{v}=\mn_f \cdot \bc \bm{v}.
\een

For a matrix field, the operators $\ccurl$ and $\ddiv$ apply by row, and $\ccurl^{*}$ and $\ddiv^{*}$ apply by column to produce a matrix field and vector field, respectively. Given a plane $f$ with unit normal vector $\mn_f$, denote 
\be
\label{Qf}
\mathcal{Q}_f:=\bm{I}-\bm{n}_{f}\bm{n}_{f}^{\intercal}
\ee
the orthogonal projection onto $f$.

Given a matrix valued function $\bsi$, define the symmetric projection $\mathcal{Q}_{f,\ssym}\bsi:=\ssym(\mathcal{Q}_f \bsi)$, 
and for symmetric $\bsi$, define $\Lambda_{f}(\bsi): f\rightarrow \mathcal{Q}_f\ms \mathcal{Q}_f$ by
\be
\label{lambdaf}
\Lambda_{f}(\bsi):=\mathcal{Q}_f(2\varepsilon(\bsi\mn_f)-\frac{\partial\bsi}{\partial\mn_{f}}) \mathcal{Q}_f.
\ee

Assume that $D\subseteq \mathbb{R}^{3}$ is a bounded and topologically trivial domain. 
Let standard notation $H^{m}(D;X)$ denote the Sobolev space consisting of functions within domain $D$,\ taking values in space $X$,\ and with all derivatives of order at most $m$ square-integrable.\ In the case $m=0$,\ set $H^0(D;X) = L^2(D;X)$.\ 
The $L^2$ scalar product over $D$ is denoted as $(\cdot,\cdot)_{D}$, $\Vert{\cdot}\Vert_{0,D}$ denotes the $L^2$ norm over a set $D$, and $\Vert{\cdot}\Vert_0$ abbreviates $\Vert{\cdot}\Vert_{0,\Omega}$. Similarly, let $C^{m}(D;X)$ denote the space of $m$-times continuously differentiable functions, taking values in $X$.\ In this paper,\ $X$ could be $\mathbb{M},\mathbb{S},\mathbb{T},\mathbb{R}$ or $\mathbb{R}^3$.\ Define
\begin{equation*}
	\begin{aligned}
		H(\ssym\ccurl, D ; \mathbb{T}):=&\left\{\bsi \in L^{2}(D ; \mathbb{T}): \ssym\ccurl \bsi\in L^{2}(D ; \mathbb{S})\right\},\\
		H(\ddiv\ddiv, D ; \mathbb{S}):=&\left\{\bsi \in L^{2}(D ; \mathbb{S}): \ddiv\ddiv \bsi \in L^{2}(D ; \mathbb{R})\right\}.
	\end{aligned}
\end{equation*} 
Let $P_{k}(D;X)$ denote the space of polynomials of degree no more than $k$ on $D$,\ taking values in the space $X$. Let $Q_{k}(D; X)$ denote the space of polynomials on $D$ of degree  no more than $k$ for each variable, taking values in the space $X$. It will be convenient to define
\ben
P_{k_1}(x_i)\cdot P_{k_2}(x_j) \cdot P_{k_3}(x_l):=\left\{\sum_{0 \leq s_1\leq k_{1},0 \leq s_2\leq k_{2}
	\atop
	0\leq s_3\leq k_{3}}c_{s_1 s_2 s_3}\, x_{i}^{s_1}x_{j}^{s_2}x_{l}^{s_3} :\, c_{s_1 s_2 s_3}\in \mathbb{R}\right\}.
\een
Here $s_1$, $s_2$, $s_3$, $k_1$, $k_2$, and $k_3$ are nonnegative integers, and $\{i, j, l\}$ is a permutation of $\{1,2,3\}$. In particular, for $\mathbf{x}=(x_1,x_2,x_3)^\intercal\in D$, denote
\ben
P_{k_1 k_2 k_3}(D):=P_{k_1}(x_1)\cdot P_{k_2}(x_2)\cdot P_{k_3}(x_3).
\een

Given two spaces $\mathcal{Q}_1(D)$ and $\mathcal{Q}_2(D)$ with $\mathcal{Q}_2(D)\subseteq \mathcal{Q}_1(D)$, let $\mathcal{Q}_1(D)\slash \mathcal{Q}_2(D)$ denote the space of $q\in \mathcal{Q}_1(D)$ such that $(q, p)_D=0$ for any $p\in \mathcal{Q}_2(D)$. 

For $\mathbf{x}=(x, y, z)^{\intercal}\in D$, denote
\be
RM(D):=\left\{\begin{pmatrix}
	c_{1}-c_{4}y-c_{5}z\\
	c_{2}+c_{4}x-c_{6}z\\
	c_{3}+c_{5}x+c_{6}y\\
\end{pmatrix} : c_{1}, c_{2}, c_{3}, c_{4}, c_{5}, c_{6} \in \mathbb{R}
\right\}
\ee
as the 6-dimensional space of the rigid motions. Besides, $$RT(D):=\{a\mathbf{x}+\bm{b}:\, a\in\mb, \bm{b}\in \mb^{3}\}$$ with ${\rm{dim}}\, RT(D)=4$ is the shape function space of the lowest order Raviart-Thomas element\cite{brezzi2012mixed}. If the domain $D$ is clear, the spaces $RM(D)$ and $RT(D)$ are recorded as $RM$ and $RT$, respectively.

Let $\operatorname{ker}(\mathscr{D})$ be the kernel space of operator $\mathscr{D}$. Let $\overline{n}$ denote the congruence of integer $n$ with modulus 3, namely, 
\be
\label{overline}
\overline{n}\equiv n(\mmod 3), ~~~~~\overline{n}=1, 2, 3.
\ee

For $\bm{u}=(u_{ij})_{3\times 3}$, let $\mathscr{P}_n\bm{u}:=(u_{11}, u_{22}, u_{33})^{\intercal}$ denote the normal stress and $\mathscr{P}_t\bm{u}:= (u_{12}, u_{13}, u_{23})^{\intercal}$ as the shear stress.  

\section{The finite element spaces on cuboid grids}
Let $\mathcal{T}_{\Box}$ be a cuboid grid of the domain $\Omega\subseteq\mathbb{R}^{3}$.
This section constructs finite element subspaces $V_{k,\Box}\subseteq H^1(\Omega;\mathbb{R}^3)$, $\mathcal{U}_{k,\Box}\subseteq H(\ssym\ccurl,\Omega;\mathbb{T})$, and $\Sigma_{k,\Box}\subseteq H(\ddiv\ddiv,\Omega;\mathbb{S})$ on cuboid grids. 
It will be proved that these finite element subspaces form an exact discrete divdiv complex of \eqref{eq:DisComp}.

In the endeavor to acquire the exact finite element divdiv complex of \eqref{eq:DisComp} on $\mathcal{T}_{\Box}$, one of the challenging tasks is to construct the $H(\ddiv\ddiv,\Omega;\ms)$ conforming finite element space $\Sigma_{k,\Box}$. 
It is not easy to get a unisolvent $H(\ddiv\ddiv,\Omega;\ms)$ conforming finite element space on cuboid grids 
by strengthening the continuity of the normal component of the divergences of the symmetric matrix valued functions of the $H(\ddiv,\Omega;\ms)$ conforming finite element on cuboid grids as in \cite{2020arXiv201002638H,2021arbitrary}. To attack this difficulty, the de Rham complex and the divdiv complex which consist of polynomials over hexahedrons are proposed below to determine the shape function space. Besides, the continuity as that of the $H(\ddiv\ddiv,\Omega;\ms)$ conforming finite element space in \cite{2020arXiv201002638H} is imposed sufficiently to construct the DOFs. Then the $H(\ddiv\ddiv,\Omega;\ms)$ conforming finite element space $\Sigma_{k,\Box}$ follows, which is actually a subspace of $H(\ddiv\ddiv,\Omega;\ms)\cap H(\ddiv,\Omega;\ms)$. As a consequence, there is corresponding enhanced smoothness for the remaining finite element spaces $\mathcal{U}_{k,\Box}$ and $V_{k,\Box}$ in the discrete divdiv complex.


For a cuboid partition $\mathcal{T}_{\Box}$, on any cube $K\in \mathcal{T}_{\Box}$ with $K=(x_{a}, x_{b})\times (y_{a}, y_{b})\times (z_{a}, z_{b})$, let
\be
\label{def:bKi}
\begin{split}
&b_{K,x}:=\frac{(x-x_{a})(x-x_{b})}{h_{K,x}^{2}}, ~~~~b_{K,y}:=\frac{(y-y_{a})(y-y_{b})}{h_{K,y}^{2}},\\ &b_{K,z}:=\frac{(z-z_{a})(z-z_{b})}{h_{K,z}^{2}}.
\end{split}
\ee
Here $h_{K,x}:=x_{b}-x_{a}$, $h_{K,y}:=y_{b}-y_{a}$, and $h_{K,z}:=z_{b}-z_{a}$. Let $b_{K}:=b_{K,x}b_{K,y}b_{K,z}$.  And $b_{K, x_{i}}$ is abbreviated as $b_{K,i}$ in the text for $i=1,2,3$. Let $E_i(K)\subseteq \mathscr{E}(K)$ be the set of all edges of $K$ which parallel to the $x_i$-axis and let $F_i(K)\subseteq \mathscr{F}(K)$ be the set of all faces of $K$ that perpendicular to the $x_i$-axis. Let $F_i(K)\cup F_j(K)$ denote the set of all faces of $K$ which are either in $F_{i}(K)$ or in $ F_{j}(K)$. Let $F_i\subseteq \mathscr{F}$ be the set of all faces of $\mathcal{T}_{\Box}$ that perpendicular to the $x_i$-axis.
Hereafter, if not special specified, $\{i, j, l\}$ is a permutation of $\{1,2,3\}$.

\subsection{Polynomial complexes}
Two polynomial complexes are established in this subsection. They will be used in the construction of the discrete complex of \eqref{eq:DisComp} on $\mathcal{T}_{\Box}$ below. In particular, the shape function spaces of $V_{k,\Box}$, $\mathcal{U}_{k,\Box}$, and $\Sigma_{k,\Box}$ are determined, respectively.

Given a topological trivial domain $D\subseteq \mathbb{R}^3$, define 
\be
\label{spacedef:N}
\begin{split}
	&M_{[k]}(D;\mb^{3}):= P_{k-1, k, k}(D)\times P_{k, k-1, k}(D)\times P_{k, k, k-1}(D),\\
	&V_{[k]}(D;\mb^{3}):= P_{k, k-1, k-1}(D)\times P_{k-1, k, k-1}(D)\times P_{k-1, k-1, k}(D).
\end{split}
\ee
The exact smooth de Rham complex \cite{2020Complexes} reads
\be
\label{smooth:de}
\mathbb{R}\stackrel{\subseteq}{\longrightarrow}C^{\infty}(D;\mathbb{R})\stackrel{{\nabla}}{\longrightarrow} C^{\infty}(D;\mb^{3})\stackrel{{\bc}}{\longrightarrow}C^{\infty}(D;\mb^{3})\stackrel{{\rd}}{\longrightarrow} C^{\infty}(D;\mathbb{R}){\longrightarrow} 0.
\ee

\begin{lm}
	\label{complex:de}
	On any topological trivial domain $D\subseteq \mathbb{R}^3$, the following polynomial sequence 
	\ben
	~~~~~~~~~\mb\stackrel{\subseteq}{\longrightarrow}Q_{k-1}(D;\mathbb{R})\stackrel{{\nabla}}{\longrightarrow} M_{[k-1]}(D;\mb^{3})\stackrel{{\bc}}{\longrightarrow}V_{[k-1]}(D;\mb^{3})\stackrel{{\rd}}{\longrightarrow} Q_{k-2}(D;\mathbb{R}){\longrightarrow} 0
	\een
	 is an exact de Rham complex.
\end{lm}

\begin{proof}
	For any $q\in Q_{k-1}(D;\mathbb{R})$ with $\nabla q=0$, it follows that $q\in\mb$. This is
	\be
	\label{complex:de:temp:1}
	\rk (\nabla)\cap Q_{k-1}(D;\mathbb{R})=\mb.
	\ee
If $\bm{u}\in M_{[k-1]}(D;\mb^{3})$ with $\bc\bm{u}=0$, then the exactness of \eqref{smooth:de} shows that there exists some $q\in C^{\infty}(D;\mathbb{R})$ such that $\bm{u}=\nabla q$. Since the $i$-th component of $\nabla q$ belongs to $P_{k}(x_{\overline{i-1}})\cdot P_{k-1}(x_{i})\cdot P_{k}(x_{\overline{i+1}})$ for $i=1,2,3$, simple arguments lead to $q\in Q_{k-1}(D;\mathbb{R})$. Therefore 
	\be
	\label{complex:de:temp:2}
	\rk(\bc)\cap M_{[k-1]}(D;\mb^{3})=\nabla Q_{k-1}(D;\mathbb{R}).
	\ee
	To show the exactness, it suffices to show that the decomposition 
	\be
	\label{directsum:1}
	V_{[k-1]}(D;\mb^{3})=\bc\, M_{[k-1]}(D;\mb^{3})\oplus \mathbf{x}Q_{k-2}(D;\mathbb{R})
	\ee
	is a direct sum. To this end, given $\bm{\phi}\in \bc\, M_{[k-1]}(D;\mb^{3})\cap \mathbf{x}Q_{k-2}(D;\mathbb{R})$, 
	$\bm{\phi}=\bc\bm{u}$ for some $\bm{u}\in M_{[k-1]}(D;\mb^{3})$. On the other hand, $\bm{\phi}= \mathbf{x}q$ for some $q\in Q_{k-2}(D;\mathbb{R})$. Since $\ddiv (\mathbf{x}Q_{k-2})= Q_{k-2}$, it follows immediately that $q=0$. Hence $\bm{\phi}=0$ and $ \bc\, M_{[k-1]}(D;\mb^{3})\cap \mathbf{x}Q_{k-2}(D;\mathbb{R})=\{{0}\}$ follows. 
	
	Furthermore, it is obvious that $\bc\, M_{[k-1]}(D;\mb^{3})$ and $\mathbf{x}Q_{k-2}(D;\mathbb{R})$ are subspaces of $V_{[k-1]}(D;\mb^{3})$. Note that \eqref{complex:de:temp:1}--\eqref{complex:de:temp:2} lead to 
	\ben
	\begin{split}
		&{\rm{dim}}\bc M_{[k-1]}(D;\mb^{3})+{\rm{dim}}\, \mathbf{x}Q_{k-2}(D;\mathbb{R})\\&={\rm{dim}}\, M_{[k-1]}(D;\mb^{3})-{\rm{dim}}\, Q_{k-1}(D;\mathbb{R})+1+{\rm{dim}}\, Q_{k-2}(D;\mathbb{R})\\
		&=3k^{2}(k-1)-k^{3}+1+(k-1)^{3}=3k^{3}-6k^{2}+3k={\rm{dim}}V_{[k-1]}(D;\mb^{3}).
	\end{split}
	\een
	This proves \eqref{directsum:1}.
	
	If $\bm{\phi}\in V_{[k-1]}(D;\mb^{3})$ with $\rd\bm{\phi}=0$, then \eqref{directsum:1} shows that there exists some $\bm{u}\in M_{[k-1]}(D;\mb^{3})$ such that $\bm{\phi}=\bc\bm{u}$. This implies \, ${\rk}(\rd)\cap V_{[k-1]}(D;\mb^{3})=\bc M_{[k-1]}(D;\mb^{3})$. Moreover, by \eqref{directsum:1}, it is straightforward to get $${\rd}V_{[k-1]}(D;\mb^{3})=Q_{k-2}(D;\mathbb{R}).$$ This concludes the proof.
\end{proof}
\begin{re}
	Actually, for nonnegative integers $k_{1}$, $k_{2}$, and $k_{3}$, the exactness of a more general polynomial de Rham complex 
	\begin{small}
		\be
		\label{general:derham}
		\begin{split}
		\mb\stackrel{\subseteq}{\longrightarrow}P_{k_{1}+1,k_{2}+1,k_{3}+1}(D)&\stackrel{{\nabla}}{\longrightarrow} M_{[k_{1}k_{2}k_{3}]}(D;\mb^{3})\\
		&\stackrel{{\bc}}{\longrightarrow}V_{[k_{1}k_{2}k_{3}]}(D;\mb^{3})\stackrel{{\rd}}{\longrightarrow} P_{k_{1},k_{2},k_{3}}(D){\longrightarrow} 0
		\end{split}
		\ee
	\end{small}on a contractible domain $D$ can be proved through similar arguments as that of Lemma \ref{complex:de}. Here $$M_{[k_{1}k_{2}k_{3}]}(D;\mb^{3}):=P_{k_{1},k_{2}+1,k_{3}+1}(D)\times P_{k_{1}+1,k_{2},k_{3}+1}(D)\times P_{k_{1}+1,k_{2}+1,k_{3}}(D),$$ and $$V_{[k_{1}k_{2}k_{3}]}(D;\mb^{3}):=P_{k_{1}+1,k_{2},k_{3}}(D)\times P_{k_{1},k_{2}+1,k_{3}}(D)\times P_{k_{1},k_{2},k_{3}+1}(D).$$
\end{re}

Define
	\be
	\label{sigma:poly}
	\begin{split}
		\Sigma_{[k]}(D;\ms):=\{&\bsi\in H^{1}(D;\ms):  \\
		&\mathscr{P}_n\bsi\in P_{k, k-2, k-2}(D)\times P_{k-2, k, k-2}(D)\times P_{k-2, k-2, k}(D),\\
		& \mathscr{P}_t\bsi \in  P_{k-1, k-1, k-2}(D)\times P_{k-1, k-2, k-1}(D)\times P_{k-2, k-1, k-1}(D)\},\\
	\end{split}
	\ee
	and
\be
\label{U:3d:def}
\begin{split}
	\mathcal{U}_{[k]}(D;\mathbb{T}):=\{&\bm{u}\in H^{1}(D;\mathbb{T}):  \mathscr{P}_n\bm{u}\in Q_{k-1}(D;\mb^{3}), \\
	&\mathscr{P}_t\bm{u} \in P_{k, k-2, k-1}(D)\times P_{k, k-1, k-2}(D)\times P_{k-1, k, k-2}(D),\\
	& \mathscr{P}_t(\bm{u}^{\intercal})\in  P_{k-2, k, k-1}(D)\times P_{k-2, k-1, k}(D)\times P_{k-1, k-2, k}(D)\}.\\
\end{split}
\ee

\begin{lm}
	\label{complex:poly3d:3}
	The following sequence is an exact polynomial divdiv complex on a topological trivial domain $D\subseteq \mb^{3}$.
	\ben
	~~~~~~~~~RT\stackrel{\subseteq}{\longrightarrow}V_{[k]}(D;\mb^{3})\xlongrightarrow[~]{\rm{dev}\ggrad} \mathcal{U}_{[k]}(D;\mathbb{T})\xlongrightarrow[~]{\rm{sym}\bc}\Sigma_{[k]}(D;\ms)\xlongrightarrow[~]{\rd\bd} Q_{k-2}(D;\mathbb{R}){\longrightarrow} 0.
	\een
\end{lm}

\begin{proof}
	Apply the exactness of \eqref{general:derham} recursively and notice that ${\rd}{\bd}\bta=0$ for any $\bta\in C^{2}(D;\mathbb{K})$. It is straightforward to see that  $${\rd}{\bd}\Sigma_{[k]}(D;\ms)={\rd}{\bd}\Sigma_{[k]}(D;\mathbb{M})=Q_{k-2}(D;\mathbb{R}).$$
	Besides, $$RT\subseteq V_{[k]}(D;\mb^{3})\cap {\rk}({\rm{dev}}\ggrad)\subseteq H^{1}(D;\mb^{3})\cap {\rk}({\rm{dev}}\ggrad)=RT.$$
	
	For any $\bm{u}\in \mathcal{U}_{[k]}(D;\mathbb{T})$ with ${\rm{sym}}\bc \bm{u}=0$, \eqref{divdivComp} shows that there exists $\bm{v}\in H^{1}(D;\mb^{3})$, such that $$\bm{u}={\rm{dev}\, \ggrad}\, \bm{v}=\nabla\bm{v}-\frac{1}{3}({\rd}\bm{v})I.$$ Simple calculations lead to 
	\ben
	3\bc \bm{u}=-\bc ({\rd} \bm{v}I)=\operatorname{mspn}\, \nabla(\ddiv \bm{v}).
	\een
	This implies ${\rd}\bm{v}\in Q_{k-1}(D;\mathbb{R})$. This and Lemma \ref{complex:de} lead to $\bm{v}\in V_{[k]}(D;\mb^{3})$. Hence $${\rk}({\rm{sym}}\bc)\cap \mathcal{U}_{[k]}(D;\mathbb{T})={\rm{dev}\, \ggrad}\,  V_{[k]}(D;\mb^{3}).$$
	
	In addition, a direct calculation leads to 
	\ben
	\begin{split}
		&{\rm{dim}}\, ({{\rm{sym}}\, {\bc}}\, {\mathcal{U}}_{[k]}(D;\mathbb{T}))={\rm{dim}}\, {\mathcal{U}}_{[k]}(D;\mathbb{T})-{\rm{dim}}\, ({\rm{dev}\, {\ggrad}}{V}_{[k]}(D;\mb^{3}))\\
		&={\rm{dim}}\, {\mathcal{U}}_{[k]}(D;\mathbb{T})-{\rm{dim}}\, {V}_{[k]}(D;\mb^{3})+4=5k^3 - 3k^2 - 6k + 4,
	\end{split}
	\een
	and
	\ben
	\begin{split}
		&{\rm{dim}}\, (\rk({\rd}\, {\bd})\cap \Sigma_{[k]}(D;\ms))={\rm{dim}}\, {\Sigma}_{[k]}(D;\ms)-{\rm{dim}}({\rd}\, {\bd}{\Sigma}_{[k]}(D;\ms))\\
		&={\rm{dim}}\, {\Sigma}_{[k]}(D;\ms)-{\rm{dim}}\, {Q}_{k-2}(D;\mathbb{R})=5k^3 - 3k^2 - 6k + 4.
	\end{split}
	\een
	This shows $${{\rm{sym}}{\bc}}\, {\mathcal{U}}_{[k]}(D;\mathbb{T})=\rk({\rd}\, {\bd})\cap{\Sigma}_{[k]}(D;\ms).$$
	
	Thus the exactness of the complex follows.
\end{proof}

Given $K\in \mathcal{T}_{\Box}$, the vectorial polynomial space $V_{[k]}(K;\mb^{3})$, the traceless matrix valued polynomial space $\mathcal{U}_{[k]}(K;\mathbb{T})$, and the symmetric matrix valued polynomial space $\Sigma_{[k]}(K;\ms)$ are determined as the shape function spaces of the $H^1(\Omega;\mathbb{R}^3)$ conforming finite element space $V_{k,\Box}$, the $H(\ssym\ccurl,\Omega;\mathbb{T})$ conforming finite element space $\mathcal{U}_{k,\Box}$, and the $H(\ddiv\ddiv,\Omega;\mathbb{S})$ conforming finite element space $\Sigma_{k,\Box}$, respectively. In light of the previous discussion, it suffices to construct the DOFs by imposing the continuity requirements appropriately for each finite element space appearing in the discrete divdiv complex of \eqref{eq:DisComp}.

\subsection{$H(\ddiv \ddiv)$-conforming finite element space}
The $H(\ddiv \ddiv,\Omega;\ms)$ conforming finite element space $\Sigma_{k,\Box}$ is constructed in this subsection. 
The shape function space is 
\ben
	\begin{split}
		\Sigma_{[k]}(K;\ms):=\{&\bsi\in H^{1}(K;\ms):  \\
		&\mathscr{P}_n\bsi\in P_{k, k-2, k-2}(K)\times P_{k-2, k, k-2}(K)\times P_{k-2, k-2, k}(K),\\
		& \mathscr{P}_t\bsi \in  P_{k-1, k-1, k-2}(K)\times P_{k-1, k-2, k-1}(K)\times P_{k-2, k-1, k-1}(K)\}
	\end{split}
\een
	defined in \eqref{sigma:poly} with $k\geq 3$ and $K\in \mathcal{T}_{\Box}$. The global regularity $$\Sigma_{k,\Box}\subseteq H(\ddiv \ddiv,\Omega;\ms)\cap H(\ddiv,\Omega;\ms)$$ stemming from  \cite{2020arXiv201002638H} is imposed, which can also be noted as $$\Sigma_{k,\Box}\subseteq\{\bsi\in H(\ddiv,\Omega;\mathbb{S}):\ddiv\bsi\in H(\ddiv,\Omega;\mathbb{R}^3)\}.$$
	Actually, this is a sufficient but not necessary continuity condition for a piecewise smooth function to be in the Sobolev space $H(\ddiv\ddiv,\Omega;\ms)$.
	Sufficient and necessary conditions are presented in \cite[Proposition 3.6]{MR3925478}.

To construct the DOFs of $\Sigma_{k,\Box}\subseteq H(\ddiv \ddiv,\Omega;\ms)\cap H(\ddiv,\Omega;\ms) $, a polynomial bubble function space is defined by 
\ben
\begin{split}
	\mathring{\Sigma}_{[k]}(K;\ms):=\{\bsi \in \Sigma_{[k]}(K;\ms): &\sigma_{ii}|_{f}=0, \partial_i\sigma_{ii}|_f=0~ \text{for all}~f\in F_i(K),\\
	&\sigma_{ij}|_f=0~\text{for all}~f\in F_i(K)\cup F_j(K) \}.
\end{split}
\een

\begin{remark}
Recall $b_{K,i}$ from \eqref{def:bKi}. Any matrix valued function $\bsi=(\sigma_{ij})_{3\times 3}\in \mathring{\Sigma}_{[k]}(K;\ms)$ implies $\sigma_{ij}=\sigma_{ji}$, and for $k\geq 4$, it holds 
	\ben
	\begin{array}{cc}
	\sigma_{ii}=b_{K,i}^2q, ~~\text{for some}~~q\in P_{k-2}(x_{\overline{i-1}})\cdot P_{k-4}(x_i)\cdot P_{k-2}(x_{\overline{i+1}}),\\
	\sigma_{ij}=b_{K,i}b_{K,j}q, ~~\text{for some}~~q\in P_{k-3}(x_{i})\cdot P_{k-3}(x_j)\cdot P_{k-2}(x_{l}),
	\end{array}
	\een 
	when $k=3$, $\sigma_{ii}=0$, and $\sigma_{ij}=b_{K,i}b_{K,j}q$, for some $q\in P_1(x_l)$. 
	\end{remark}
\begin{The}
	\label{unisolvence}
	Given cube $K\in \mathcal{T}_{\Box}$, the symmetric matrix valued polynomial ${\bsi}\in \Sigma_{[k]}( {K};\ms)$ with $k\geq 3$ can be uniquely determined by the following conditions:
	\begin{subequations}\label{divdivDof3}
	\begin{align}
		( {\sigma}_{i\, \overline{i-1}}, {q})_{ {e}} \quad& ~~   \text{for all}~~  {q}\in P_{k-2}( {e};\mathbb{R}), e\in E_{\overline{i+1}}(K), \label{divdivDof3:1}\\
		(\sigma_{i\, \overline{i-1}},\partial_{\overline{i-1}}q)_{f},(\sigma_{i\, \overline{i+1}},\partial_{\overline{i+1}}q)_{f} \quad&~~\text{for all}~~ q\in Q_{k-2}(f;\mathbb{R}),f\in F_{i}(K), \label{divdivDof3:2}\\
		( {\sigma}_{ii},  {q})_{ {f}}, ( {\partial}_{i} {\sigma}_{ii},  {q})_{ {f}} \quad&~~  \text{for all}~~  {q}\in Q_{k-2}( {f};\mathbb{R}),f\in F_{i}(K),\label{divdivDof3:3}\\
		( {\bsi},   {\bta})_{ {K}} \quad&~~\text{for all}~~ {\bta}\in  \mathring{\Sigma}_{[k]}(K;\ms).\label{divdivDof3:6}
	\end{align}
\end{subequations}
	Here $i=1, 2, 3$ and the definition of $\overline{n}$ is from \eqref{overline}.
\end{The}
\begin{remark}
\label{re:sigma:1}
The continuity of $\Sigma_{k,\Box}$ is characterized by the continuity of $\bsi\mn_{f}$ and $\mn^{\intercal}_{f}\ddiv\bsi$ across each interface $f$ as in \cite{2020arXiv201002638H}. The continuity of $\bsi\mn_{f}$ is imposed by the DOFs \eqref{divdivDof3:2} and the first part of \eqref{divdivDof3:3}. The DOFs \eqref{divdivDof3:1}--\eqref{divdivDof3:2} and the second part of \eqref{divdivDof3:3} ensure the continuity of $\mn^{\intercal}_{f}\ddiv\bsi$ across each interface $f$. 
The interior moments of \eqref{divdivDof3:6} are derived by the vanishing traces and the symmetry of the tensor. 
\end{remark}
\begin{proof}
	The number of conditions in \eqref{divdivDof3} is 
	\ben 
	\begin{split}
		&12(k-1)+12(k-2)(k-1)+12(k-1)^{2}+3(k-1)^2(k-3)\\
		&+3(k-2)^2(k-1)=6k^{3}-6k^{2}-3k+3\\
		&=3(k+1)(k-1)^{2}+3k^{2}(k-1)={\rm{dim}}\Sigma_{[k]}( {K};\ms).\\
	\end{split}
	\een
	It suffices to prove $ {\bsi}=0$ provided that \eqref{divdivDof3} vanish. 
	The vanishing of \eqref{divdivDof3:3} leads to $\sigma_{ii}|_f=0$ and $\partial_i\sigma_{ii}|_{f}=0$ for all $f\in F_i(K)$. Besides, according to \eqref{divdivDof3:2}, $\sigma_{ij}|_{f}=0$ holds for all $f\in F_i(K)\cup F_j(K)$. 
	 Thus \eqref{divdivDof3:6} completes the proof.
\end{proof}
In three dimensions, the corresponding global space for the $H(\ddiv\ddiv,\Omega;\ms)$ conforming element $\Sigma_{k, \Box}$ is defined by
\ben
\begin{split}
	\Sigma_{k, \Box}:=&\{\bsi\in H({\rd}\, {\bd},\Omega;\ms): \bsi|_{K}\in \Sigma_{[k]}(K;\ms)~\text{for all}~~K\in\mathcal{T}_{\Box},\\
	&\text{all of the degrees of freedom \eqref{divdivDof3} are single-valued}\}.
\end{split}
\een

\subsection{$H(\operatorname{sym}\operatorname{curl})$-conforming finite element space} This subsection constructs the $H(\operatorname{sym} \bc,\Omega;\mathbb{T})$ conforming finite element space $\mathcal{U}_{k,\Box}$ which satisfies the inclusion $\ssym\ccurl \mathcal{U}_{k,\Box}\subseteq \Sigma_{k,\Box}$.
The shape function space is 
\ben
\begin{split}
	\mathcal{U}_{[k]}(K;\mathbb{T}):=\{&\bm{u}\in H^{1}(K;\mathbb{T}):  \mathscr{P}_n\bm{u}\in Q_{k-1}(K;\mb^{3}), \\
	&\mathscr{P}_t\bm{u} \in P_{k, k-2, k-1}(K)\times P_{k, k-1, k-2}(K)\times P_{k-1, k, k-2}(K),\\
	& \mathscr{P}_t(\bm{u}^{\intercal})\in  P_{k-2, k, k-1}(K)\times P_{k-2, k-1, k}(K)\times P_{k-1, k-2, k}(K)\}
\end{split}
\een
defined in \eqref{U:3d:def} with $k\geq 3$ and $K\in \mathcal{T}_{\Box}$.
To construct the DOFs, the main difficulty arises in the characterization of the continuity of functions in $\mathcal{U}_{k,\Box}$.  
Based on the sufficient continuity condition deduced below, a subspace of $H(\operatorname{sym} \bc,\Omega;\mathbb{T})$ with enhanced regularity will be introduced to attack the difficulty. 
 
Firstly, a conclusion that draws heavily on the traceless property of matrix valued functions is presented in the following lemma.
\begin{lemma}
\label{tracelessproperty}
	Given a traceless matrix valued function $\bm{u}\in H(\ccurl,\Omega;\mathbb{T})$, its column vector $\bm{u}^{\intercal}\in H(\ddiv,\Omega;\mathbb{T})$.
\end{lemma}
\begin{proof}
	Let $\bm{u} = (u_{ij})_{3\times 3}\in H(\ccurl,\Omega;\mathbb{T})$. Note that 
	\begin{equation*}
		\begin{split}
			{\partial_{1} u_{11}}&+{\partial_{2} u_{21}}+{\partial_{3} u_{31}} = {-\partial_{1}(u_{22}+u_{33})}+{\partial_{ 2}u_{21}}+{\partial_{3} u_{31}} \\&= ({\partial_{2} u_{21}} -{\partial_{1} u_{22}}) +({\partial_{3} u_{31}}-{\partial_{1} u_{33}}) \in L^2(\Omega;\mathbb{R}).
		\end{split}
	\end{equation*}
	\textcolor{black}{Similar arguments} imply $\ddiv^{*}\bm{u}\in L^2(\Omega;\mathbb{R}^3)$. This completes the proof.
\end{proof}

For any traceless matrix valued function $\bm{u}\in \mathcal{U}_{k,\Box}$, let $\bsi=\ssym \ccurl \bm{u}$. It follows from the inclusion condition $\ssym \ccurl\, \mathcal{U}_{k,\Box}\subseteq \Sigma_{k,\Box}$ that the symmetric matrix valued function $\bsi$ ought to be in the $H(\ddiv\ddiv,\Omega;\ms)$ conforming finite element space $ \Sigma_{k,\Box}$ constructed in Section 3.2. This implies the continuity requirements for $\bsi\mn_{f}$ and $\ddiv\bsi\cdot\mn_{f}$ across each interface $f$. The following lemma presents two identities. They motivate the sufficient conditions of enhanced smoothness for $\bm{u}$.
\begin{lemma}
	\label{sigma-u}
	Given a matrix valued function $\bm{u}\in C^{2}(\Omega;\mathbb{M})$, let $\bm{\sigma} = \ssym\ccurl \bm{u}$. It holds
	\begin{subequations}\label{remark:sigma}
	\begin{align}
		&\bm{\sigma}\mn_{f}  = (\ccurl\bm{u})\mn_f-\frac{1}{2}(\ddiv^{*}\bm{u})\times\mn_{f},\label{re:1}\\	
		&(\ddiv\bm{\sigma})\cdot\mn_{f} = \frac{1}{2}\ccurl(\ddiv^{*}\bm{u})\cdot\mn_{f}.\label{re:2}
	\end{align}
       \end{subequations}
	Here $f\in \mathscr{F}$ is a facet with the unit normal vector $\mn_f$. 
\end{lemma}
\begin{proof}
Elementary calculations lead to 
\ben
\begin{split}
&\bm{\sigma}\mn_{f} = (\ccurl\bm{u} -\frac{1}{2}\operatorname{mspn}(\ddiv^{*}\bm{u}))\mn_{f} = (\ccurl\bm{u})\mn_f-\frac{1}{2}(\ddiv^{*}\bm{u})\times\mn_{f},\\
&(\ddiv\bm{\sigma})\cdot\mn_{f} = \frac{1}{2}\ddiv(\ccurl\bm{u})^{\intercal}\cdot\mn_{f}=\frac{1}{2}\ccurl(\ddiv^{*}\bm{u})\cdot\mn_{f}.
\end{split}
\een
\end{proof}
\begin{remark}
\label{u:continuity:condition}
From Lemma \ref{sigma-u}, given a traceless matrix valued function $\bm{u}\in \mathcal{U}_{k,\Box}$, it follows from  $\operatorname{sym}\bc \bm{u}\in \Sigma_{k,\Box}$ that $\bm{u}\times \mn_{f}$ and $(\ddiv^{\ast}\bm{u})\times \mn_{f}$ have to be continuous across each interface $f$. This implies that $\bm{u}\in H(\operatorname{curl},\Omega;\mathbb{T})$, and then Lemma \ref{tracelessproperty} leads to $\bm{u}^{\intercal}\in H(\operatorname{div},\Omega;\mathbb{T})$, which means $\ddiv^{\ast}\bm{u}\in L^2(\Omega;\mathbb{R}^3)$.
\end{remark}

Define
\begin{equation*}
 H^{*}(\ssym\ccurl,\Omega;\mathbb{T}):=\{\bm{u}\in H(\ccurl,\Omega;\mathbb{T}): \bm{u}^{\intercal}\in H(\ddiv,\Omega;\mathbb{T}), \ddiv^{*}\bm{u}\in H(\ccurl,\Omega;\mathbb{R}^3)\},
\end{equation*} 
which is a subspace of $H(\ssym\ccurl,\Omega;\mathbb{T})$. Lemma \ref{tracelessproperty} leads to the simplification  
\begin{equation}
\label{def:symcurlast}
	H^{*}(\ssym\ccurl,\Omega;\mathbb{T}):= \{\bm{u}\in H(\ccurl,\Omega;\mathbb{T}): \ddiv^{*}\bm{u}\in H(\ccurl,\Omega;\mathbb{R}^3)\}.
\end{equation}

Motivated by Lemma \ref{sigma-u}, the global regularity $\mathcal{U}_{k, \Box}\subseteq H^{*}(\ssym\ccurl,\Omega;\mathbb{T})$ is imposed. 
To construct the DOFs of  $ \mathcal{U}_{k, \Box}\subseteq H^{*}(\ssym\ccurl,\Omega;\mathbb{T})$, a polynomial bubble function space is defined by 
\be
\label{U:3d:def:b}
\begin{split}
	\mathring{\mathcal{U}}_{[k]}( {K};\mathbb{T}):=\{\bm{u}\in \mathcal{U}_{[k]}( {K};\mathbb{T}):  &{u}_{ii}|_{ {f}}=0~~\text{for all}~~ {f}\in \mathscr{F}( {K}), \\
	& {u}_{ij}|_{f}=0 ~~\text{for all}~~ {f}\in F_i({K})\cup F_l(K),\\
	&\partial_{i} {u}_{ij}|_{ {f}}=0 ~~\text{for all}~~ {f}\in F_i({K})\}. \\
\end{split}
\ee
\begin{remark}
	\label{re:symcurl}
	Recall $b_{K}:=b_{K,i}b_{K,j}b_{K,l}$ and $b_{K,i}$ from \eqref{def:bKi}. For a tensor function $\bm{u}=(u_{ij})_{3\times 3}\in\, \mathring{\mathcal{U}}_{[k]}( {K};\mathbb{T})$, it holds that
	$${u}_{ii}=b_K q\quad\text{for some}~~q\in  Q_{k-3}( {K};\mathbb{R}),$$
	$$u_{ij}=b_{K,i}^2 b_{K, l}\,  q\quad\text{for some}~~ q\in P_{k-4}(x_i)\cdot P_{k-2}(x_{j})\cdot P_{k-3}(x_l),$$
	in which $\{i, j, l\}$ is a permutation of $\{1, 2, 3\}$. 
\end{remark}
\begin{The}
	Given cube $ {K}\in \mathcal{T}_{\Box}$, the traceless matrix valued polynomial $ \bm{u}\in \mathcal{U}_{[k]}( {K};\mathbb{T})$ with $k\geq 3$ can be uniquely determined by the following conditions:
	\begin{subequations}\label{3dsymcurlDof}
		\begin{align}
			{u}_{ii}( {a})\quad&\text{for all}~~ {a}\in \mathscr{V}( {K}), i=1,2, \label{3dsymcurlDof:1}\\
			( {u}_{ii}, {q})_{ {e}}\quad&\text{for all}~~ {q}\in P_{k-3}( {e};\mathbb{R}),  {e}\in \mathscr{E}( {K}), i=1,2,\label{3dsymcurlDof:2}\\
			( {u}_{ji},  {q})_{ {e}}, ( {\partial}_{j} {u}_{ji}, {q})_{ {e}} \quad&\text{for all}~~ {q}\in P_{k-2}( {e};\mathbb{R}),e\in E_i(K), \label{3dsymcurlDof:3}\\
			(u_{ij}, \partial_{i}(\partial_{i} q))_{f}\quad&\text{for all}~~q\in Q_{k-2}(f;\mathbb{R}), f\in F_l(K), \label{3dsymcurlDof:4-5}\\
			(u_{ij}, \partial_{l}q)_{f}, (\partial_{i}u_{ij},\partial_l q)_{f}\quad&\text{for all}~~q\in Q_{k-2}(f;\mathbb{R}), f\in F_i(K),  \label{3dsymcurlDof:6-7}\\
			( {u}_{ii}, {q})_{ {f}}\quad&\text{for all}~~ {q}\in Q_{k-3}( {f};\mathbb{R}),  {f}\in \mathscr{F}( {K}), i=1,2,\label{3dsymcurlDof:8}\\
			( \bm{u},  \bm{w})_{ {K}}\quad&\text{for all}~~ \bm{w}\in \mathring{\mathcal{U}}_{k}( {K};\mathbb{T}).\label{3dsymcurlDof:9}
		\end{align}
	\end{subequations}
\end{The}
\begin{remark}
The continuity of $\mathcal{U}_{k,\Box}$ is characterized by the continuity of $\bm{u}\times \mn_f$ and $(\ddiv^{\ast}\bm{u})\times \mn_f$ across each interface $f$. The DOFs \eqref{3dsymcurlDof:1}-\eqref{3dsymcurlDof:3} and  \eqref{3dsymcurlDof:6-7}-\eqref{3dsymcurlDof:8} ensure the continuity of $\bm{u}\times \mn_f$ across each interface $f$. The DOFs \eqref{3dsymcurlDof:3}--\eqref{3dsymcurlDof:6-7} lead to the continuity of $(\ddiv^{\ast}\bm{u})\times \mn_f$  across each interface $f$.
The interior moments of \eqref{3dsymcurlDof:9} are derived by the vanishing of \eqref{3dsymcurlDof:1}--\eqref{3dsymcurlDof:8} and the traceless property of $\bm{u}$.
\end{remark}
\begin{proof}
	First off, it is easy to compute $${\rm{dim}}\, \mathring{\mathcal{U}}_{[k]}( {K};\mathbb{T})=2(k-2)^{3}+6(k-3)(k-1)(k-2)=8k^3 - 48k^2 + 90k - 52.$$
	
	Thus the number of degrees of freedom \eqref{3dsymcurlDof} is 
	\ben
	\begin{split}
		&16+24(k-2)+48(k-1)+12(k-3)(k-1)+24(k-1)(k-2)+12(k-2)^2\\
		&+8k^3 - 48k^2 + 90k - 52=8k^3-6k=2k^3+6(k+1)(k-1)k={\operatorname{dim}}\, {\mathcal{U}}_{[k]}( {K};\mathbb{T}).\\
	\end{split}
	\een
	
	Then it remains to prove $ \bm{u}=0$ provided that \eqref{3dsymcurlDof:1}--\eqref{3dsymcurlDof:9} vanish. According to \eqref{3dsymcurlDof:1}--\eqref{3dsymcurlDof:2} and \eqref{3dsymcurlDof:8}, there exists some $q\in Q_{k-3}( {K};\mathbb{R})$ such that ${u}_{ii}=b_K q$. Besides, the DOFs \eqref{3dsymcurlDof:3}--\eqref{3dsymcurlDof:6-7} lead to  
	\ben
	u_{ij}=b_{K,i}^2 b_{K, l}\,  q, ~~\text{for some}~~ q\in P_{k-4}(x_i)\cdot P_{k-2}(x_{j})\cdot P_{k-3}(x_l).
	\een
	Hence the DOFs \eqref{3dsymcurlDof:9} complete the proof.
\end{proof}

In three dimensions, the corresponding global space for the $H({\rm{sym}\bc},\Omega;\mathbb{T})$ conforming finite element $\mathcal{U}_{k, \Box}$ is defined by
\ben
\begin{split}
	\mathcal{U}_{k, \Box}:=&\{\bm{u}\in H({\rm{sym}\bc},\Omega;\mathbb{T}):\bm{u}|_{K}\in \mathcal{U}_{[k]}(K;\mathbb{T}) ~\text{for all}~~K\in\mathcal{T}_{\Box},\\
	&\text{all of the degrees of freedom \eqref{3dsymcurlDof} are single-valued}\}.
\end{split}
\een

\begin{remark}
	\label{re:symcurl:inclu}
 On cuboid grids, for any $\bm{u}\in \mathcal{U}_{k,\Box}$, let $\bsi=\operatorname{sym}\bc \bm{u}$. Then
\ben
\begin{split}
\sigma_{ii}&=\partial_{\overline{i-1}}u_{i\, \overline{i+1}}-\partial_{\overline{i+1}}u_{i\, \overline{i-1}},\\
\sigma_{ij}&=\frac{1}{2}(\partial_{l}u_{jj}-\partial_{l}u_{ii}+\partial_{i}u_{il}-\partial_{j}u_{jl}), i<j.
\end{split} 
\een 
It can be derived from the DOFs \eqref{3dsymcurlDof:1}--\eqref{3dsymcurlDof:8} that $[u_{ii}]_f=0$ for all interfaces $f\in \mathscr{F}$, $[u_{ij}]_f=0$ and $[\partial_{i}u_{ij}]_{f}=0$ for all interfaces $f\in {F}_i\cup {F}_{l}$. For each interface $f\in {F}_i$, all of the tangential derivatives of $[u_{ij}]_f$ and $[\partial_i u_{ij}]_f$ vanish. This implies $[\sigma_{ii}]_f=0$ and $[\partial_i\sigma_{ii}]_f=0$ for all interfaces $f\in {F}_i$. Besides, a combination of the vanishing of the tangential derivatives of $[u_{ii}]_f$ on each interface $f\in {F}_i$ leads to $[\sigma_{ij}]_f=0$ for all $f\in {F}_i\cup {F}_j$. This with Lemma \ref{complex:poly3d:3} shows $\bsi\in \Sigma_{k,\Box}$.
\end{remark}

\subsection{$H^1$-conforming finite element space}This subsection constructs the $H^1(\Omega;\mathbb{R}^{3})$ conforming finite element space $V_{k,\Box}$ which satisfies the inclusion $\operatorname{dev}\, \operatorname{grad} V_{k,\Box}\subseteq \mathcal{U}_{k,\Box}$. 
The shape function is $$V_{[k]}(K;\mathbb{R}^3):= P_{k, k-1, k-1}(K)\times P_{k-1, k, k-1}(K)\times P_{k-1, k-1, k}(K)$$ defined in \eqref{spacedef:N}  with $k\geq 3$ and $K\in \mathcal{T}_{\Box}$.

\begin{remark}
\label{re:3}
It follows from the inclusion condition $\operatorname{dev}\, \operatorname{grad} V_{k,\Box}\subseteq \mathcal{U}_{k,\Box}$ that, for a vector valued function $\bm{v}\in V_{k,\Box}$, $\bm{u}=\operatorname{dev}\, \operatorname{grad}\bm{v}$ has to be in $\mathcal{U}_{k,\Box}$. This and Remark \ref{u:continuity:condition} imply
the additional requirement of the continuity of $(\ddiv^{\ast}\bm{u})\times \mn_{f}$ across all the interfaces $f$ for $\bm{u}$. Elementary calculations lead to
\ben
\ddiv^{\ast}(\operatorname{dev}\, \operatorname{grad} \bm{v})\times \mn_{f}=\ddiv^{*}(\ggrad \bm{v}-\frac{1}{3}(\ddiv\bm{v})\bm{I})\times\bm{n}_{f}= \frac{2}{3}\ggrad(\ddiv\bm{v})\times\bm{n}_{f}.
\een
Thus the continuity of the tangential derivatives of $\ddiv\bm{v}$ on each interface $f$ will be imposed on $V_{k,\Box}$.
\end{remark}
Define 
\be
\label{H1regularity}
H^1(\ddiv,\Omega;\mathbb{R}^3):=\{\bm{v}\in H^1(\Omega;\mathbb{R}^3): \ddiv \bm{v}\in H^1(\Omega;\mathbb{R})\}.
\ee
Remark \ref{re:3} motivates the global regularity $V_{k,\Box}\subseteq H^{1}(\ddiv,\Omega;\mathbb{R}^3)$. To construct the DOFs of $V_{k,\Box}\subseteq H^1(\ddiv,\Omega;\mathbb{R}^3)$ on each cube $ {K}\in \mathcal{T}_{\Box}$, define a polynomial bubble function space as follows,
\ben
\begin{split} 
\mathring{V}_{[k]}( {K};\mb^{3}):=\{ {\bm{v}}\in V_{[k]}( {K};\mb^{3}):  &{\bm{v}}|_{ {f}}=0\quad \text{for all}~{f}\in \mathscr{F}( {K})\\
&{\partial}_{i} {v}_{i}|_{ {f}}=0\quad\text{for all}~{f}\in F_i( {K})\}.
\end{split}
\een
\begin{remark}
	Recall $b_{K}:=b_{K,i}b_{K,j}b_{K,l}$ and $b_{K,i}$ from \eqref{def:bKi}. For a vector $\bm{v}=(v_1, v_2, v_3)^{\intercal}\in \mathring{V}_{[k]}( {K};\mb^{3})$, it holds that 
$${v}_{i}=b_{K}b_{K,i}\, q,~~\text{for some}~~q\in P_{k-4}( {x}_{i})\cdot P_{k-3}( {x}_{\overline{i-1}})\cdot P_{k-3}( {x}_{\overline{i+1}}),\, i=1, 2, 3.$$
\end{remark}
\begin{The}
	Given cube $ {K}\in\mathcal{T}_{\Box}$, a vectorial polynomial $ \bm{v}\in V_{[k]}( {K};\mb^{3})$ with $k\geq 3$ can be uniquely determined by the following conditions:
	\begin{subequations}\label{3dh1Dof}
	\begin{align}
		{v}_{i}( {a}), \partial_{i} {v}_{i}( {a})\quad& \text{for all}~~  {a}\in \mathscr{V}( {K}),\label{3dh1Dof:1}\\
		( {v}_{i}, {q})_{ {e}}\quad&\text{for all}~~  {q}\in P_{k-4}( {e};\mathbb{R}), e\in E_i(K), \label{3dh1Dof:2}\\
		( {v}_{i}, {q})_{ {e}}, ( {\partial}_{i} {v}_{i}, {q})_{ {e}}\quad&\text{for all}~~  {q}\in P_{k-3}( {e};\mathbb{R}),e\in E_j(K)\cup E_l(K),\label{3dh1Dof:3}\\
		( {v}_{i},  {q})_{ {f}},( {\partial}_{i} {v}_{i},  {q})_{ {f}}\quad& \text{for all}~~  {q}\in Q_{k-3}( {f};\mathbb{R}), f\in F_i(K), \label{3dh1Dof:4}\\
		( {v}_{{i}}, \partial_{i}{q})_{ {f}}\quad&\text{for all} ~~ {q}\in Q_{k-3}( {f};\mathbb{R}),f\in F_j(K)\cup F_l(K), \label{3dh1Dof:5}\\
		( {\bm{v}},  {\bm{w}})_{ {K}}\quad&\text{for all} ~~  {\bm{w}}\in \mathring{V}_{[k]}( {K};\mb^{3}).\label{3dh1Dof:7}
	\end{align}
\end{subequations}
\end{The}
\begin{remark}
The DOFs \eqref{3dh1Dof:1}--\eqref{3dh1Dof:4} ensure the continuity of $\bm{v}$ across each interface $f$. The DOFs \eqref{3dh1Dof:1} and \eqref{3dh1Dof:3}--\eqref{3dh1Dof:5} lead to the continuity of $\ddiv\bm{v}$ across each interface $f$.
The interior moments of \eqref{3dh1Dof:7} are derived by the vanishing of \eqref{3dh1Dof:1}--\eqref{3dh1Dof:5}.
\end{remark}
\begin{proof}
	Note that ${\rm{dim}}\mathring{V}_{[k]}( {K};\mb^{3})=3(k-3)(k-2)^{2}$. The number of the DOFs \eqref{3dh1Dof} is 
	\ben
	\begin{split}
		&48+12(k-3)+48(k-2)+12(k-2)^2+12(k-2)(k-3)+3(k-3)(k-2)^{2}\\
		&=3k^{3}+3k^{2}=3(k+1)k^{2}={\rm{dim}}V_{[k]}( {K};\mb^{3}).\\
	\end{split}
	\een
	The vanishing of \eqref{3dh1Dof:1}--\eqref{3dh1Dof:5} leads to 
	$ {v}_{i}=b_{K}b_{K,i}\, q,~~\text{for some}~~q\in P_{k-4}( {x}_{i})\cdot P_{k-3}( {x}_{\overline{i-1}})\cdot P_{k-3}( {x}_{\overline{i+1}})$. Then it follows from \eqref{3dh1Dof:7} that $ {\bm{v}}=0$.
\end{proof}
In three dimensions, the corresponding global space for the $H^{1}(\Omega;\mb^{3})$ conforming finite element $V_{k, \Box}$ is defined by
\ben
\begin{split}
	V_{k, \Box}:=&\{\bm{v}\in H^1(\Omega;\mathbb{R}^3):\bm{v}|_{K}\in V_{[k]}(K;\mb^{3})~\text{for all}~~K\in\mathcal{T}_{\Box},\\
	&\text{all of the degrees of freedom \eqref{3dh1Dof} are single-valued}\}.
\end{split}
\een

\begin{remark}
	\label{re:devgrad:inclu}
	For any $\bm{v}\in V_{k,\Box}$, let $\bm{u}=\operatorname{dev}\, \operatorname{grad}\, \bm{v}$, then it holds that 
	\ben
	\begin{split}
		u_{ii}&=\frac{2}{3}\partial_{i}v_i-\frac{1}{3}(\partial_{j}v_j+\partial_{l}v_{l}),\\
		u_{ij}&=\partial_{j}v_i.
	\end{split}	\een 
Here ${i, j, l}$ is a permutation of $\{1,2,3\}$. On each cube $K\in \mathcal{T}_{\Box}$, Lemma \ref{complex:poly3d:3} shows $\bm{u}|_{K}\in \mathcal{U}_{[k]}(K;\mathbb{T})$. Furthermore, the continuity of $\bm{v}$ as well as $\partial_{i}v_i$ across each interface $f\in \mathscr{F}$ implies $[u_{ii}]_{f}=0$ for all $f\in \mathscr{F}$, $[u_{ij}]_f=0$ and $[\partial_i u_{ij}]_f=0$ for all $f\in {F}_i\cup {F}_l$. Thus, together with Lemma \ref{complex:poly3d:3}, $\bm{u}\in \mathcal{U}_{k,\Box}$ follows. 
\end{remark}

\section{The finite element spaces on tetrahedral grids} 
Let $\mathcal{T}_{\triangle}$ be a tetrahedral grid of the domain $\Omega\subseteq \mathbb{R}^{3}$.
This section constructs finite element subspaces $V_{k+2,\triangle}\subseteq H^1(\Omega;\mathbb{R}^3)$, \, $\mathcal{U}_{k+1,\triangle}\subseteq H(\ssym\ccurl,\Omega;\mathbb{T})$, $\Sigma_{k,\triangle}\subseteq H(\ddiv\ddiv,\Omega;\mathbb{S})$, and $Q_{k-2,\triangle}\subseteq L^2(\Omega;\mathbb{R})$ on $\mathcal{T}_{\triangle}$. It will be proved that these finite element subspaces form an exact discrete divdiv complex of \eqref{eq:DisComp}.

Since it is difficult to construct a finite element complex of \eqref{eq:DisComp} by directly using the $H(\ddiv\ddiv,\Omega;\ms)$ conforming finite element space from \cite{2020arXiv201002638H}, a modification by moving some degrees of freedom on faces to vertices has been made here. This leads to the  $H(\ddiv\ddiv,\Omega;\ms)$ conforming finite element space $\Sigma_{k,\triangle}$. As it can be seen below, this modification in some sense enhances the regularity at vertices of the $H(\ddiv\ddiv,\Omega;\ms)$ conforming finite element space $\Sigma_{k,\triangle}$. This brings the corresponding enhancement of regularity at vertices of the remaining finite element spaces for the discrete divdiv complex, namely, $\mathcal{U}_{k+1,\triangle}$ and $V_{k+2,\triangle}$.

Given tetrahedron $K\in\mathcal{T}_{\triangle}$, let $\mathbf{x}_{1}$, $\mathbf{x}_{2}$, $\mathbf{x}_{3}$, $\mathbf{x}_{4}$ be its vertices. Let $\lambda_{i}$ denote the $i$-th barycentric coordinate of $K$, and $f_{i}$ be the face of $K$ opposite to $\mathbf{x}_{i}$. For each $f\in \mathscr{F}$, let $\lambda_{f,i}$ be the $i$-th barycentric coordinate with respect to $f$. 

\subsection{$H(\operatorname{div}\operatorname{div})$-conforming finite element space}
This subsection constructs the $H(\ddiv\ddiv,\Omega;\ms)$ conforming finite element space $\Sigma_{k,\triangle}$, which is a modification of that defined in 
\cite[Section 2.4]{2020arXiv201002638H} with enhanced regularity at vertices. This space consists of piecewise polynomials of degree not greater than $k$ which is a subspace of the space $\{\bsi\in H(\ddiv,\Omega;\mathbb{S}):\ddiv\bsi\in H(\ddiv,\Omega;\mathbb{R}^3)\}$.

The following two lemmas are needed for the construction of $\Sigma_{k,\triangle}$.
 \begin{lemma}[{\cite[Theorem 5.1]{MfemB}}~] 
	\label{l1}
	Given $K\in \mathcal{T}_{\triangle}$, suppose $\bm{\psi}\in P_{k}(K;\mathbb{R}^3)$ with $\ddiv\bm{\psi}=0$ and $\bm{\psi}\cdot\mn_{f}|_{f}=0$ for all faces $f\in \mathscr{F}(K)$. Then there exists some $\bm{v}\in W_{k+1}(K;\mathbb{R}^3)$ such that
	\begin{equation*}
	\bm{\psi}=\ccurl\bm{v},
	\end{equation*}
	where $W_{k+1}(K;\mathbb{R}^3)$ is defined by
	\begin{equation*}
	W_{k+1}(K;\mathbb{R}^3) :=\{\bm{\phi}\in P_{k+1}(K;\mathbb{R}^3):\bm{\phi}\times \mn_{f}|_{f}=0~~\text{for all}~f\in \mathscr{F}(K)\}.
	\end{equation*}
\end{lemma}
\begin{lemma}[{\cite[Lemma 7.3]{MR2398766}}~] 
	\label{l2}Given $K\in\mathcal{T}_{\triangle}$, suppose that $\bm{\tau}\in P_{k}(K;\mathbb{S})$ with $\ddiv\bm{\tau}=0$ and $\bm{\tau}\mn_{f}|_{f} =0$ for all faces $f\in \mathscr{F}(K)$. Then there exists $\bm{u}\in M_{k+2}(K;\mathbb{S})$ such that
	\begin{equation*}
	\bm{\tau} =\ccurl\ccurl^{*}\bm{u},
	\end{equation*}
	with
	\begin{equation*}
	M_{k+2}(K;\mathbb{S}):=\{\bm{\tau}\in P_{k+2}(K;\mathbb{S}):\Lambda_{f}(\bm{\tau})|_{f}=\mathcal{Q}_f\bta\mathcal{Q}_f|_{f}=0~~\text{ for all}~f\in\mathscr{F}(K)\},
	\end{equation*}
where $\Lambda_{f}(\bm{\tau}):=\mathcal{Q}_f(2\varepsilon(\bta\mn_f)-\frac{\partial\bta}{\partial\mn_{f}}) \mathcal{Q}_f$ and $\mathcal{Q}_f:=\bm{I}-\bm{n}_{f}\bm{n}_{f}^{\intercal}$ are defined in \eqref{lambdaf} and \eqref{Qf} above, respectively.	
\end{lemma}

Define the following two spaces
\begin{equation}
\label{W}
\mathcal{W}_{k-1}(K;\mathbb{R}^3):=\ccurl W_{k}(K;\mathbb{R}^3)/RM,
\end{equation}
and
\begin{equation}
\label{M}
\mathcal{M}_{k}(K;\mathbb{S}):= \ccurl\ccurl^{*}M_{k+2}.
\end{equation}
The dimensions of these two spaces are \cite[Theorem 2.15]{2020arXiv201002638H}:
\begin{equation*}
\operatorname{dim}\mathcal{W}_{k-1}(K;\mathbb{R}^3) = \frac{2k^3-3k^2-5k-12}{6},
\end{equation*}
and \cite[Theorem 7.2]{MR2398766}:
\begin{equation*}
\operatorname{dim}\mathcal{M}_{k}(K;\mathbb{S}) = \frac{k^3-3k^2-4k+12}{2}.
\end{equation*}

The DOFs of the $H(\ddiv\ddiv,\Omega;\ms)$ finite element space $\Sigma_{k,\triangle}$ are stated in the following theorem.
\begin{theorem}
\label{sigma:dofs}
	Given tetrahedron $K\in \mathcal{T}_{\triangle}$, the symmetric matrix valued polynomial $\bsi\in P_{k}(K;\ms)$ with $k\geq 3$ can be uniquely determined by the following conditions:
	\begin{subequations}\label{DOF:Sig}
		\begin{align}
			\bsi(a),\ddiv\bsi(a) \quad &\text{for all}~a\in\mathscr{V}(K),\label{DOF:Sig1}\\
			(\bm{t}_{e}^{T}\bsi\bm{n}_{e}^{\pm}, q)_{e},((\bm{n}_{e}^{\pm})^{T}\bsi\bm{n}_{e}^{\pm}, q)_{e} \quad&\text{for all}~q\in P_{k-2}(e;\mathbb{R}),~e\in \mathscr{E}(K),\label{DOF:Sig2}\\
			(\bsi\bm{n}_{f}, \bm{q})_{f} \quad &\text{for all}~\bm{q}\in P_{k-3}(f;\mathbb{R}^3),~f\in \mathscr{F}(K),\label{DOF:Sig3}\\
			(\ddiv\bsi\cdot\bm{n}_{f}, {q})_{f}  \quad&\text{for all}~{q}\in
			{\widetilde{P}}_{k-1}(f;\mathbb{R}),~f\in\mathscr{F}(K),\label{DOF:Sig4}\\
			(\bsi, \nabla^2q)_{K} \quad &\text{for all}~q\in P_{k-2}(K;\mathbb{R}),\label{DOF:Sig5}\\
			(\bsi,\nabla\bm{q})_{K} \quad&\text{for all}~\bm{q}\in \mathcal{W}_{k-1}(K;\mathbb{R}^3),\label{DOF:Sig6}\\
			(\bsi,\bm{\tau})_{K} \quad&\text{for all}~\bm{\tau}\in \mathcal{M}_{k}(K;\mathbb{S}).\label{DOF:Sig7}
		\end{align}
	\end{subequations}
Here $ {\widetilde{P}}_{k-1}(f;\mathbb{R}):=\{{q}\in P_{k-1}(f;\mathbb{R}):{q}\, \text{vanishes at all vertices of }f\}$.
\end{theorem}
\begin{remark}
The DOFs \eqref{DOF:Sig1}--\eqref{DOF:Sig3} ensure the continuity of $\bsi\mn_{f}$ across each interface $f$, and the DOFs \eqref{DOF:Sig1} as well as \eqref{DOF:Sig4} lead to $\ddiv\bsi\in H(\ddiv;\mathbb{R}^{3})$.
\end{remark}
\begin{proof}
	The proof is similar as that in \cite[Theorem 2.15]{2020arXiv201002638H}. The number of all degrees of freedom \eqref{DOF:Sig} is
	\ben
	\begin{split}
	&24+12+30(k-1)+6(k-1)(k-2)+4k(k+1)-12+\frac{(k+1)k(k-1)}{6}-4\\
	&+\frac{2k^3-3k^2-5k-12}{6}+\frac{k^3-3k^2-4k+12}{2}=(k+1)(k+2)(k+3)=\operatorname{dim}P_{k}(K;\ms).
	\end{split}
	\een 
	It suffices to prove if the degrees of freedom \eqref{DOF:Sig} vanish for some $\bsi\in P_{k}(K;\ms)$, then $\bsi=0$. For $q\in P_{k-2}(K)$, an integration by parts and \eqref{DOF:Sig1}, \eqref{DOF:Sig3}--\eqref{DOF:Sig5} lead to 
	\ben
	(\ddiv \ddiv \bsi, q)_{K} = (\bsi,\nabla^{2}q)_{K}-\sum_{f\in \mathscr{F}(K)}(\bsi\mn_f,\nabla q)_{f}+\sum_{f\in \mathscr{F}(K)}(\ddiv\bsi\cdot\mn_f,q)_{f}=0.
	\een
	Thus it follows that $\ddiv \ddiv\bsi=0$. Therefore, Lemma \ref{l1}, \eqref{DOF:Sig1}, and \eqref{DOF:Sig4} show that there exists some $\bm{v}\in W_{k}(K;\mathbb{R}^{3})$, such that $\ddiv\bsi=\ccurl \bm{v}$. Furthermore, \eqref{DOF:Sig1}--\eqref{DOF:Sig3} yield 
	\be
	\label{ela_}
	\bsi\mn_f|_{f}=0~~~~\text{for all}~f\in\mathscr{F}(K). 
	\ee
	This and an integration by parts imply the orthogonality of $\ddiv\bsi$ and the rigid motion space $RM$. Then $\ccurl \bm{v}=0$ can be derived from \eqref{DOF:Sig6} and \eqref{ela_}. That is $\ddiv \bsi=0$.  Hence Lemma \ref{l2} shows $\bsi=\ccurl \ccurl^{\ast}\bm{\omega}$ for some $\bm{\omega}\in M_{k+2}(K;\ms)$. Finally it follows from \eqref{DOF:Sig7} that $\bsi=0$.  
\end{proof}
In three dimensions, the corresponding global space for the $H({\rd}\, {\bd},\Omega;\ms)$ conforming element $\Sigma_{k, \triangle}$ is defined by
\be
\label{glo-sigma}
\begin{split}
	\Sigma_{k, \triangle}:=&\{\bsi\in H({\rd}\, {\bd},\Omega;\ms): \bsi|_{K}\in P_k(K;\ms)~\text{for all}~~K\in\mathcal{T}_{\triangle},\\
	&\text{all of the degrees of freedom \eqref{DOF:Sig} are single-valued}\}.
\end{split}
\ee

\begin{remark}
The $H(\ddiv\ddiv,\Omega;\ms)$ conforming finite element space in \cite[(3.1)]{2020arXiv201002638H} for solving the fourth order problem can be replaced by  $\Sigma_{k,\triangle}$ in \eqref{glo-sigma} with $k\geq 3$. Using the similar arguments as those in \cite[Section 3]{{2020arXiv201002638H}}, the well-posedness of this new mixed finite element can be proved. 
However, to acquire the exact finite element divdiv complex of \eqref{eq:DisComp} on $\mathcal{T}_{\triangle}$ in this paper, $k$ has to be greater than $3$, since the finite element spaces $\mathcal{U}_{k+1,\triangle}\subseteq H(\ssym\ccurl,\Omega;\mathbb{T})$ and $V_{k+2,\triangle}\subseteq H^1(\Omega;\mathbb{R}^3)$ constructed below are well-defined only for $k\geq 4$.
\end{remark}

\subsection{$H(\operatorname{sym}\operatorname{curl})$-conforming finite element space}
This subsection constructs the $H(\ssym\ccurl,\Omega;\mathbb{T})$ conforming finite element space $\mathcal{U}_{k+1,\triangle}$, which consists of piecewise polynomials of degree not greater than $k+1$ with $k\geq 4$. As mentioned in Section 3.3, the difficulty is to characterize the continuity of functions in the space $\mathcal{U}_{k+1,\triangle}$. Besides the $H(\ssym\ccurl,\Omega;\mathbb{T})$ conformity, the inclusion condition $\ssym \ccurl \mathcal{U}_{k+1,\triangle} \subseteq \Sigma_{k,\triangle}$ is also required. Recall
$$H^{\ast}(\ssym \ccurl,\Omega;\mathbb{T}):= \{\bm{u}\in H(\ccurl,\Omega;\mathbb{T}): \ddiv^{*}\bm{u}\in H(\ccurl,\Omega;\mathbb{R}^3)\}$$ from \eqref{def:symcurlast}. Motivated by Lemma \ref{sigma-u}, the global regularity $$\mathcal{U}_{k+1,\triangle}\subseteq H^{\ast}(\ssym \ccurl,\Omega;\mathbb{T})$$ is also imposed herein. 

To construct the DOFs of $\mathcal{U}_{k+1,\triangle}\subseteq H^{\ast}(\ssym \ccurl,\Omega;\mathbb{T})$ on $\mathcal{T}_{\triangle}$, a polynomial bubble function space is defined by  
\be
\begin{split}
\label{pdiv}
\mathring{\mathcal{P}}^{\ddiv}_{k+2}(K;\mathbb{R}^{3}):=\{&\bm{q}\in \lambda_{1}\lambda_{2}\lambda_{3}\lambda_{4}P_{k-2}(K;\mathbb{R}^{3}): \\
&\ddiv \bm{q}|_{f}=0~~\text{for all}~f\in \mathscr{F}(K)\}.
\end{split}
\ee
Note that for any $\bm{q}\in \mathring{\mathcal{P}}^{\ddiv}_{k+2}(K;\mathbb{R}^{3})$, there exists $\bm{p}\in P_{k-2}(K;\mathbb{R}^{3})$ with $\bm{p}\cdot\mn_f|_{f}=0$ for all $f\in \mathscr{F}(K)$ such that 
$\bm{q}=\lambda_{1}\lambda_{2}\lambda_{3}\lambda_{4}\bm{p}$.

\begin{theorem}
	\label{th3-8}
	Given tetrahedron $K\in \mathcal{T}_{\triangle}$, the traceless matrix valued polynomial $\bm{u}\in P_{k+1}(K;\mathbb{T})$ with $k\geq 4$ can be uniquely determined by the following conditions:
	\begin{subequations}\label{DOF:Lam}
		\begin{align}
			\bm{u}(a),\nabla\bm{u}(a),\nabla(\ddiv^{*}\bm{u})(a) \quad &\text{for all}~a\in \mathscr{V}(K), \label{DOF:Lam1} \\
			(\bm{u},\bm{\omega})_{e} \quad &\text{for all}~ \bm{\omega}\in P_{k-3}(e;\mathbb{T}), e\in \mathscr{E}(K), \label{DOF:Lam2}\\
			(\ddiv^{*}\bm{u},\bm{q})_{e} \quad &\text{for all}~ \bm{q}\in P_{k-4}(e;\mathbb{R}^3), e\in \mathscr{E}(K), \label{DOF:Lam3}\\
			(\bm{t}_{e}^{T}\bsi\bm{n}_{e}^{\pm},q)_{e},((\bm{n}_{e}^{\pm})^{T}\bsi\bm{n}_{e}^{\pm}, q)_{e} \quad &\text{for all}~q\in  P_{k-2}(e;\mathbb{R}),~e\in\mathscr{E}(K), \label{DOF:Lam4}\\
			(\bm{u}\times \bm{n}_{f},\nabla_{f}\bm{q})_{f} \quad &\text{for all}~\bm{q}\in P_{k-3}(f;\mathbb{R}^3),~f\in\mathscr{F}(K), \label{DOF:Lam5}\\
			(\bm{u}\times \bm{n}_{f},\ccurl_{f}\bm{q})_{f} \quad & \text{for all}~\bm{q}\in \mathring{\mathcal{P}}_{k+2}(f;\mathbb{R}^3),~f\in\mathscr{F}(K), \label{DOF:Lam6}\\
			((\ddiv^{*}\bm{u})\times\bm{n}_{f},\bm{n}_{f}\times\bm{q}\times\bm{n}_{f})_{f} \quad&\text{for all}~\bm{q}\in P_{k-3}(f;\mathbb{R}^3),~f\in\mathscr{F}(K), \label{DOF:Lam7}\\
			(\bsi,\nabla\bm{q})_{K} \quad&\text{for all}~\bm{q}\in \mathcal{W}_{k-1}(K;\mathbb{R}^3), \label{DOF:Lam8}\\
			(\bsi,\bm{\omega})_{K} \quad &\text{for all}~\bm{\omega}\in \mathcal{M}_{k}(K;\mathbb{S}), \label{DOF:Lam9}\\
			(\bm{u},\operatorname{dev}\ggrad\bm{q})_{K} \quad &\text{for all}~\bm{q}\in \mathring{\mathcal{P}}^{\ddiv}_{k+2}(K;\mathbb{R}^{3}).\label{DOF:Lam10}	
		\end{align}
	\end{subequations}
	Here $\bm{n}_{e}^{\pm}$ are two linearly independent normal vectors of the edge $e$, $\bm{t}_{e}$ is the unit tangential vector of edge $e$, $\mathring{\mathcal{P}}_{k+2}(f;\mathbb{R}^{3}):=(\lambda_{f,1}\lambda_{f,2}\lambda_{f,3})^2P_{k-4}(f;\mathbb{R}^3)$, $\mathcal{W}_{k-1}(K;\mathbb{R}^3)$ is defined in \eqref{W}, $\mathcal{M}_{k}(K;\mathbb{S})$ is defined in \eqref{M},
	and $\bsi:=\ssym\ccurl \bm{u}$.
\end{theorem}
\begin{remark}
The DOFs \eqref{DOF:Lam1}--\eqref{DOF:Lam2} and  \eqref{DOF:Lam4}--\eqref{DOF:Lam6} lead to the continuity of $\bm{u}\times\mn_{f}$ across each interface $f$. The DOFs \eqref{DOF:Lam1}, \eqref{DOF:Lam3}, and \eqref{DOF:Lam7} ensure the continuity of $(\ddiv^{\ast}\bm{u})\times\mn_{f}$ across each interface $f$.
\end{remark}
\begin{proof}
      The number of DOFs given in \eqref{DOF:Lam} is
      \ben
      \begin{split}
      &32+96+36+48(k-2)+18(k-3)+30(k-1)+6(k-1)(k-2)-12\\
      &+6(k-2)(k-3)+4(k-1)(k-2)+\frac{2k^3-3k^2-5k-12}{6}+\frac{k^3-3k^2-4k+12}{2}\\
      &+\frac{(k+1)k(k-1)}{2}-2k(k-1)=\frac{4(k+4)(k+3)(k+2)}{3}=\operatorname{dim}P_{k+1}(K;\mathbb{T}).
      \end{split}
      \een 
      It suffices to prove that for any $\bm{u}\in P_{k+1}(K;\mathbb{T})$, if $\bm{u}$ vanishes on all the DOFs of \eqref{DOF:Lam}, then $\bm{u}=0$.
The DOFs \eqref{DOF:Lam1}--\eqref{DOF:Lam4} show that $\bm{u}(a),\nabla\bm{u}(a),\nabla(\ddiv^{*}\bm{u})$ vanish at all vertices, and $\bm{u}$, $\ddiv^{*}\bm{u}$, $(\ssym\ccurl\bm{u})\bm{n}_{e}^{\pm}$ vanish on all edges. This and \eqref{re:1} show that for any face $f$, $(\ccurl\bm{u})\cdot\bm{n}_{f} = \operatorname{rot}_{f}\bm{u}$ vanishes on $\partial f$. This, \eqref{DOF:Lam5},\eqref{DOF:Lam6} plus the DOFs of the two dimensional $H^1$-conforming finite element in \cite[(2.16)--(2.20)]{2020arXiv201002638H} show that $\bm{u}\times \bm{n}_{f}$ vanishes on $f\in \mathscr{F}(K)$. The DOFs \eqref{DOF:Lam7} show that $(\ddiv^{*}\bm{u})\times \bm{n}_{f}$ also vanishes on $f\in \mathscr{F}(K)$. 
	
	Since $\bm{\sigma} = \ssym\ccurl\bm{u}$, Lemma \ref{sigma-u} shows that $\bm{\sigma}$ vanishes at the DOFs \eqref{DOF:Sig1}--\eqref{DOF:Sig3} and $\ddiv\ddiv\bm{\sigma} = 0$ on $K$. This, \eqref{DOF:Lam8} and \eqref{DOF:Lam9} imply that $\ssym\ccurl\bm{u} =0$ on $K$. Similar arguments as those in \cite[Theorem 12]{2021arXiv210300088H} show that there exists $\bm{v}\in P_{k+2}(K;\mathbb{R}^3)$ such that
	\begin{equation*}
		\bm{u} = \operatorname{dev}\ggrad\bm{v},
	\end{equation*}
	here $\bm{v} = \lambda_1\lambda_2\lambda_3\lambda_4\bm{w}$ with $\bm{w}\in P_{k-2}(K;\mathbb{R}^3)$. A direct calculation shows
	\begin{equation*}
		\ddiv^{*}\bm{u}\times \bm{n}_{f}= \ddiv^{*}(\ggrad \bm{v}-\frac{1}{3}(\ddiv\bm{v})\bm{I})\times\bm{n}_{f}= \frac{2}{3}\ggrad(\ddiv\bm{v})\times\bm{n}_{f}.
	\end{equation*}
	Since $\ddiv^{*}\bm{u}\times \bm{n}_{f}$ vanishes in each interface $f\in \mathscr{F}(K)$, the tangential derivative of $\ddiv\bm{v}$ also vanishes on the faces. Then $\ddiv\bm{v}=0 $ on $f\in \mathscr{F}(K)$, and \eqref{DOF:Lam10} shows that $\bm{u}=0$. This concludes the proof.
\end{proof}

In three dimensions, the corresponding global space for the $H({\rm{sym}\, \bc},\Omega;\mathbb{T})$ conforming finite element $\mathcal{U}_{k+1, \triangle}$ is defined by

\ben
\begin{split}
	\mathcal{U}_{k+1, \triangle}:=&\{\bm{u}\in H({\rm{sym}\bc},\Omega;\mathbb{T}):\bm{u}|_{K}\in P_{k+1}(K;\mathbb{T}) ~\text{for all}~~K\in\mathcal{T}_{\triangle},\\
	&\text{all of the degrees of freedom \eqref{DOF:Lam} are single-valued}\}.
\end{split}
\een

\subsection{$H^1$-conforming finite element space}
This subsection constructs the $H^1(\Omega;\mathbb{R}^{3})$ conforming finite element space $V_{k+2,\triangle}$. This space consists of vectorial, globally continuous piecewise polynomials of degree not greater than $k+2$ with $k \geq 4$. Recall  $$H^1(\ddiv,\Omega;\mathbb{R}^3):=\{\bm{v}\in H^1(\Omega;\mathbb{R}^3): \ddiv \bm{v}\in H^1(\Omega;\mathbb{R})\}$$
defined in \eqref{H1regularity}.
Motivated by Remark \ref{re:3}, the global regularity $V_{k+2,\triangle}\subseteq H^1(\ddiv,\Omega;\mathbb{R}^3)$ is imposed to satisfy the inclusion condition $\operatorname{dev}\, \operatorname{grad} V_{k+2,\triangle}\subseteq \mathcal{U}_{k+1,\triangle}$. The DOFs of $V_{k+2,\triangle}$ are stated in the following theorem. 
\begin{theorem}
	\label{uniso-H1}
	Given tetrahedron $K\in \mathcal{T}_{\triangle}$, the vector valued polynomial $\bm{v}\in P_{k+2}(K;\mathbb{R}^3)$ with $k\geq 4$ can be uniquely determined by the following conditions:
	\begin{subequations}\label{DOF:V}
		\begin{align}
			\bm{v}(a),\nabla\bm{v}(a),\nabla^2\bm{v}(a),\nabla^2(\ddiv\bm{v})(a)\quad &\text{for all}~a\in \mathscr{V}(K), \label{DOF:V1} \\
			(\bm{v},\bm{q})_{e}\quad &\text{for all}~ \bm{q}\in P_{k-4}(e;\mathbb{R}^3),~e\in \mathscr{E}(K), \label{DOF:V2}\\
			(\ddiv\bm{v},q)_{e}\quad &\text{for all}~ q\in P_{k-5}(e;\mathbb{R}),~e\in \mathscr{E}(K),\label{DOF:V3}\\
			(\frac{\partial (\bm{v}\cdot\bm{n}_{e}^{+})}{\partial \bm{n}_{e}^{\pm}},q)_{e},(\frac{\partial (\bm{v}\cdot\bm{t}_{e})}{\partial \bm{n}_{e}^{\pm}},q)_{e}, (\frac{\partial (\bm{v}\cdot\bm{n}_{e}^{-})}{\partial \bm{n}_{e}^{+}}, q)_{e}\quad &\text{for all}~q\in  P_{k-3}(e;\mathbb{R}),~e\in\mathscr{E}(K), \label{DOF:V4}\\
			(\frac{\partial(\ddiv\bm{v})}{\partial \bm{n}_{e}^{\pm}},q)_{e} \quad &\text{for all}~q\in P_{k-4}(e;\mathbb{R}),~e\in\mathscr{E}(K), \label{DOF:V5}\\
			(\bm{v}, \bm{q})_{f} \quad&\text{for all}~\bm{q}\in P_{k-4}(f;\mathbb{R}^3),~f\in\mathscr{F}(K), \label{DOF:V6}\\
			(\ddiv\bm{v}, q)_{f} \quad&\text{for all}~q\in P_{k-5}(f;\mathbb{R}),~f\in\mathscr{F}(K),\label{DOF:V7}\\
			(\bm{v}, \bm{q})_{K} \quad&\text{for all}~\bm{q}\in \mathring{P}^{\ddiv}_{k+2}(K;\mathbb{R}^{3}).\label{DOF:V8}
		\end{align}	
	\end{subequations}
	Here $\bm{n}_{e}^{\pm}$ are two linearly independent normal vectors of the edge $e$, $\bm{t}_{e}$ is the unit tangential vector of edge $e$, and $\mathring{P}^{\ddiv}_{k+2}(K;\mathbb{R}^{3})$ is defined in \eqref{pdiv} above. 
\end{theorem}
\begin{remark}
The DOFs \eqref{DOF:V1}--\eqref{DOF:V2}, \eqref{DOF:V4}, and \eqref{DOF:V6} are imposed to ensure the continuity of $\bm{v}$ on each interface. The DOFs \eqref{DOF:V1}, \eqref{DOF:V3}--\eqref{DOF:V5}, and \eqref{DOF:V7} are imposed to ensure the continuity of $\ddiv\bm{v}$ on each interface.
\end{remark}
\begin{proof}
Since there are $(k+5)(k+4)(k+3)/2$ degrees of freedom defined in \eqref{DOF:V}, it suffices to prove that for any $\bm{v}\in P_{k+2}(T;\mathbb{R}^3)$, if $\bm{v}$ vanishes at all degrees of freedom in \eqref{DOF:V}, then $\bm{v}$ is identically zero.
	
	It is easy to see that $\bm{v},\nabla\bm{v},\nabla^2\bm{v},\nabla^2(\ddiv\bm{v})$ vanish at all vertices. Then the DOFs \eqref{DOF:V2},\eqref{DOF:V3},\eqref{DOF:V5} imply that $\bm{v},\ddiv\bm{v},\nabla(\ddiv\bm{v})$ vanish on all edges. Since $\ddiv\bm{v}=\frac{\partial(\bm{v}\cdot\bm{n}_{e}^{+})}{\partial\bm{n}_{e}^{+}}+\frac{\partial(\bm{v}\cdot\bm{n}_{e}^{-})}{\partial\bm{n}_{e}^{-}}+\frac{\partial(\bm{v}\cdot\bm{t}_{e})}{\partial\bm{t}_{e}}$, the DOFs \eqref{DOF:V4} show that $\nabla\bm{v}$ vanish on all edges. It then follows from \eqref{DOF:V6} and \eqref{DOF:V7} that $\bm{v},\ddiv\bm{v}$ vanish on all faces. Thus \eqref{DOF:V8} concludes the proof.
\end{proof}

In three dimensions, the corresponding global space for the $H^{1}(\Omega;\mb^{3})$ conforming finite element $V_{k+2, \triangle}$ is defined by
\ben
\begin{split}
	V_{k+2, \triangle}:=&\{\bm{v}\in H^{1}( \Omega;\mb^{3}):\bm{v}|_{K}\in P_{k+2}(K;\mb^{3})~\text{for all}~~K\in\mathcal{T}_{\triangle},\\
	&\text{all of the degrees of freedom \eqref{DOF:V} are single-valued}\}.
\end{split}
\een

\section{The discrete divdiv complex}
The finite element analogy of the divdiv complex \eqref{divdivComp} is presented in this section, and the discrete complexes consist of those conforming finite elements introduced in the previous two sections. The exactness of the finite element divdiv complexes on cuboid and tetrahedral grids is proved respectively.   
\subsection{The finite element divdiv complex on cuboid grids}
This subsection studies the finite element divdiv complex of \eqref{eq:DisComp} on cuboid grids $\mathcal{T}_{\Box}$. Remark \ref{re:symcurl:inclu} and Remark \ref{re:devgrad:inclu} show $\operatorname{dev} \operatorname{grad}\, V_{k,\Box} \subseteq \mathcal{U}_{k,\Box}$ and $\operatorname{sym}\bc\mathcal{U}_{k,\Box}\subseteq \Sigma_{k,\Box}$.
\begin{lemma}
For any $\bm{u}\in \mathcal{U}_{k,\Box}$ with $\operatorname{sym}\bc\bm{u}=0$, there exists some $ {\bm{v}}\in {V}_{k,\Box}$ such that $\bm{u}=\operatorname{dev}\operatorname{grad} \bm{v}$. 
\end{lemma}
\begin{proof}
For any $\bm{u}\in \mathcal{U}_{k,\Box}$ and $\operatorname{sym}\bc\bm{u}=0$, there exists some $\bm{v}\in H^1(\Omega;\mathbb{R}^3)$, such that $\bm{u}=\operatorname{dev}\operatorname{grad} \bm{v}$. The exactness of the polynomial complex in Lemma \ref{complex:poly3d:3} shows $ {\bm{v}}|_K\in {V}_{[k]}( {K};\mb^{3})$ on each cube $K\in \mathcal{T}_{\Box}$. To prove $\bm{v}\in V_{k,\Box}$, it suffices to verify the continuity of $\bm{v}$ as well as $\partial_i v_i$ across each interface $f\in \mathscr{F}$.
The identities
\ben
\Pi_{ {f}}( \bm{u}^{\intercal} {\mn_{f}})= {\nabla}_{ {f}}( {\bm{v}}\cdot {\mn_{f}}),~~~~~~
\mathcal{Q}_{ {f},{\rm{sym}}}( \bm{u}\times  {\mn_{f}})= {\varepsilon}_{ {f}}( {\bm{v}}\times  {\mn_{f}}) \quad\text{for all }~~ {f}\in \mathscr{F}
\een
lead to 
\ben
[{\bm{v}}\cdot {\mn_{f}}]_{ {f}}\in P_{0}( {f};\mathbb{R}), ~~~~~~ [{\bm{v}}\times  {\mn_{f}}]_{ {f}}\in RM(f)\quad\text{for all }~~ {f}\in \mathscr{F}.
\een
This implies
\ben
[{v}_{i}]_{ {f}}=c_{0}, ~~~~ [{v}_{j}]_{ {f}}=c_{2}-c_{1} {x}_{j}, ~~~~ [{v}_{l}]_{ {f}}=c_{3}+c_{1} {x}_{l} \quad \text{for all}~ f\in \mathscr{F}_i, ~~i=1, 2, 3
\een
with parameters $c_{0}$, $c_{1}$, $c_{2}$, $c_{3}$. Recall $ {u}_{ii}= {\partial}_{i} {v}_{i}-\frac{1}{3} {\rm{div}} {\bm{v}}$, 
and $ [{u}_{ii}]_{ {f}}=0$ for all $ {f}\in \mathscr{F}$. 
Thus 
\be
\label{tem:v:i}
[{\partial}_{i} {v}_{i}]_{ {f}}= [{\partial}_{j} {v}_{j}]_{ {f}}= [{\partial}_{l} {v}_{l}]_{ {f}}=0.
\ee
This shows that $c_{1}$ is zero. Therefore $ [{\bm{v}}]_{ {f}}\in P_{0}( {f};\mathbb{R})$ for all $ {f}\in \mathscr{F}$. As $RT=\rk(\rm{dev}\ggrad)$, the constraint $ {\bm{v}}(x_b,y_b,z_b)=0$ can be imposed. Hence $ [{\bm{v}}]_{ {f}}=0$ for all $ {f}\in \mathscr{F}$. This plus \eqref{tem:v:i} prove $ {\bm{v}}\in {V}_{k,\Box}$.
\end{proof}

\begin{theorem}
	\label{global-complex}
	The finite element sequence 
		\ben
	~~~~~~~~~RT\stackrel{\subseteq}{\longrightarrow}{V}_{k,\Box}\xlongrightarrow[~]{{{\rm{dev}\,  {\ggrad}}}} \mathcal{U}_{k,\Box}\xlongrightarrow[~]{{{\rm{sym}}\,  {\bc}}}{\Sigma}_{k,\Box}\xlongrightarrow[~]{{ {\rm{div}}\,  {\bd}}}{Q}_{k-2,\Box}{\longrightarrow} 0
	\een
is a divdiv complex, which is exact on a contractible domain.
\end{theorem}
\begin{proof}
	The inclusion $\rd\bd\Sigma_{k,\Box}\subseteq Q_{k-2,\Box}$ is easy to verify. It remains to prove $Q_{k-2,\Box}\subseteq \rd\bd\Sigma_{k,\Box}$. If the inclusion does not hold, then there exists some $q\in Q_{k-2,\Box}$ and $q\neq 0$, such that 
	\ben
	({\rd}\, {\bd}\bta, q)_{\Omega}=0~~~\text{for all}~~\bta\in \Sigma_{k,\Box}.
	\een
	Integration by parts and the continuity of $\bta\mn_{f}$ as well as $\mn_{f}^{\intercal}\ddiv \bta$ across each interface $f$ leads to $(\bta, \nabla^2 q)_K=0$ for any $\bta\in \Sigma_{k,\Box}$, and thus $q\in P_{1}(K;\mathbb{R})$. Take $\bta\in \Sigma_{k,\Box}$ such that its DOFs vanish except the following DOFs
	\ben
	(\partial_{i}\tau_{ii}, [q]_{f_{i}})_{f_{i}}=([q]_{f_{i}}, [q]_{f_{i}})_{f_{i}}~~~\text{for some}~~f_{i}\in {\mathscr{F}}.
	\een
	Then an integration by parts results in 
	\ben
	0=({\rd}\, {\bd}\bta, q)_{\Omega}=\sum_{f\in {\mathscr{F}}_{h}}(\bd{\bta}\cdot\mn_{f}, [q]_{f})_{f}=\|[q]_{f_{i}}\|^{2}_{0, f_{i}}.
	\een
	This shows $[q]_{f_{i}}=0$. The arbitrariness of the choice of $f_{i}\in \mathscr{F}$ leads to $q=0$. The contradiction occurs.

	Besides,
	\ben
	\begin{split}
		{\rm{dim}}\, {\rm{sym}\bc}\,  \mathcal{U}_{k, \Box}&={\rm{dim}}\, \mathcal{U}_{k,\Box}-{\rm{dim}}\, {\rm{dev}\ggrad}V_{k,\Box}={\rm{dim}}\, \mathcal{U}_{k,\Box}-{\rm{dim}}\, V_{k,\Box}+ {\rm{dim}}\,RT\\
		&=-4\#\mathscr{V}+(k+3)\#\mathscr{E}+(4k^2-10k+2)\#\mathscr{F}\\
		&\quad+(5k^{3}-27k^{2}+42k-16)\# \mathcal{T}_{h}+4,
	\end{split}
	\een
	and
	\ben
	\begin{split}
		{\rm{dim}}\, {\rk(\rd\bd)}\cap \Sigma_{k,\Box}&={\rm{dim}}\, \Sigma_{k,\Box}-{\rm{dim}}\, {\rd}{\bd}\Sigma_{k,\Box}={\rm{dim}}\, \Sigma_{k,\Box}-{\rm{dim}}\, Q_{k-2,\Box}\\
		&=(k-1)\#\mathscr{E}+(4k^2-10k+6)\#\mathscr{F}\\
		&\quad+(5k^{3}-27k^{2}+42k-20)\# \mathcal{T}_{h}.
	\end{split}
	\een
	Here $\# \mathcal{S}$ is the number of the elements in the finite set $S$. According to Euler's formula $\#\mathscr{E}+1=\#\mathscr{V}+\#\mathscr{F}-\#\mathcal{T}_{h}$, 
	\ben
	{\rm{dim}}\, {\rk(\rd\bd)}\cap \Sigma_{k,\Box}={\rm{dim}}\, {\rm{sym}\bc}\,  \mathcal{U}_{k, \Box}.
	\een
This concludes the proof. 
\end{proof}

\subsection{The finite element divdiv complex on tetrahedral grids}
This subsection derives and studies the discrete divdiv complex of \eqref{eq:DisComp} on tetrahedral grids. Since $\operatorname{dev}\ggrad H^1(\ddiv,\Omega;\mathbb{R}^3) \subseteq H^{*}(\ssym\ccurl,\Omega;\mathbb{T})$, it follows from the proof of Theorem \ref{uniso-H1} and Theorem \ref{th3-8} that $\operatorname{dev}\ggrad V_{k+2,\triangle}\subseteq \mathcal{U}_{k+1, \triangle}$ and $\ssym\ccurl\, \mathcal{U}_{k+1, \triangle}\subseteq \Sigma_{k, \triangle}$.

\begin{lemma}
\label{sur:tri:dev}
For any $\bm{u}\in \mathcal{U}_{k+1, \triangle}$ with $\ssym\ccurl\bm{u}=0$, there exists some $\bm{v}\in V_{k+2, \triangle}$ such that $\bm{u} = \operatorname{dev}\ggrad\bm{v}$.
\end{lemma}
\begin{proof}
For any $\bm{u}\in \mathcal{U}_{k+1, \triangle}$ with $\ssym\ccurl\bm{u}=0$, there exists some $\bm{v}\in H^1(\Omega;\mathbb{R}^3)$ consisting of piecewise polynomials of degree no more than $k+2$ such that $\bm{u} = \operatorname{dev}\ggrad\bm{v}$\cite[Lemma 3.2]{2020arXiv200712399C}.
 Note that
\begin{equation*}
	\ddiv^{*}\bm{u}= \ddiv^{*}(\ggrad \bm{v}-\frac{1}{3}(\ddiv\bm{v})\bm{I}) = \frac{2}{3}\ggrad(\ddiv\bm{v})\in L^2(\Omega;\mathbb{R}^3),
\end{equation*}
and $\ddiv\bm{v}\in L^2(\Omega;\mathbb{R})$, hence $\bm{v}\in H^1(\ddiv,\Omega;\mathbb{R}^3)$. Since $\bm{v}$ is a discrete function, $\bm{v}$ and $\ddiv\bm{v}$ are continuous at vertices, on edges, and in faces. Then the continuity of $\bm{u},\ddiv^{*}\bm{u}$ on edges and at vertices imply that $\nabla \bm{v},\nabla(\ddiv\bm{v})$ are also continuous on edges and at vertices. Then $\nabla_{pw}^2\bm{v} = \nabla_{pw}\bm{u} +\frac{1}{3}\nabla(\ddiv\bm{v}\bm{I})$ is continuous at all vertices. Finally, the continuity of $\nabla_{pw}(\ddiv^{*}\bm{u})$ at vertices implies that $\nabla_{pw}^2(\ddiv\bm{v})$ is continuous at vertices. Hence $\bm{v}\in V_{k+2, \triangle}$.
\end{proof}
Lemma \ref{sur:tri:dev} proves the surjection of the operator $\operatorname{dev}\ggrad: V_{k+2,\triangle}\rightarrow \operatorname{ker}(\ssym\ccurl)\cap \mathcal{U}_{k+1,\triangle}$ and the following lemma shows that the operator $\ddiv\ddiv:\Sigma_{k,\triangle} \rightarrow Q_{k-2,\triangle}$ is surjective.
\begin{lemma}
	For any $q_{h}\in Q_{k-2,\triangle}$ with $k \geq 3$, there exists some $\bm{u}_{h}\in \Sigma_{k, \triangle}$ such that $\ddiv\ddiv \bm{u}_{h} = q_{h}$.
\end{lemma}
\begin{proof}
	For any $q_{h}\in Q_{k-2,\triangle}$, there exists some nonstandard Brezzi-Douglas-Marini element $\bm{v}_{h}\in H(\ddiv,\Omega;\mathbb{R}^3)$ consisting of piecewise polynomials of degree not greater than $k-1$ such that $\ddiv \bm{v}_{h}= q_{h}$(see \cite[(2.18)]{MR2594344}). Then it suffices to prove that there exists some $\bm{u}_{h}\in \Sigma_{k, \triangle}$ so that
	\begin{equation*}
	(\ddiv\ddiv \bm{u}_{h}-\ddiv\bm{v}_{h},p_{h})=0\quad\text{for all}~p_{h}\in Q_{k-2,\triangle}.
	\end{equation*}
	The construction of $\bm{u}_{h}$ will be carried out in two steps. Recall $$ {\widetilde{P}}_{k-1}(f;\mathbb{R})=\{{q}\in P_{k-1}(f;\mathbb{R}):{q}\, \text{vanishes at all vertices of }f\}$$
	from Theorem \ref{sigma:dofs}. First set the DOFs \eqref{DOF:Sig1}, \eqref{DOF:Sig4} by
	\begin{equation*}
	\bm{u}_{h}(a) = 0,\ddiv\bm{u}_{h}(a)=\bm{v}_{h}(a) \quad \text{for all}~a\in\mathscr{V},
	\end{equation*}
	\begin{equation*}
	\begin{split}
	(\ddiv\bm{u}_{h}\cdot\bm{n}_{f},\bm{q})_f = (\bm{v}_{h}\cdot\bm{n}_{f}, \bm{q})_f\quad \text{for all}~\bm{q}\in
	{\widetilde{P}}_{k-1}(f;\mathbb{R}^3)~~\text{and}~~f\in \mathscr{F}.
	\end{split}
	\end{equation*}
	Next, set the DOFs \eqref{DOF:Sig3}, \eqref{DOF:Sig7} by
	\begin{equation*}
(\bm{u}_{h}\bm{n}_{f}, \bm{q})_f =0 \quad\text{for all}~q\in P_{k-3}(f;\mathbb{R}^3)/P_0(f;\mathbb{R}^3)\quad\text{and}~f\in \mathscr{F},
	\end{equation*}
	\begin{equation*}
	(\bm{u}_{h},\nabla^2 q)_K=(\bm{v}_{h},\nabla q)_K\quad \text{for all}~q\in P_{k-2}(K;\mathbb{R})/P_1(K;\mathbb{R})~~\text{and}~~K\in\mathcal{T}_{h}.
	\end{equation*} 
	It follows from the above DOFs that 
	\begin{equation*}
	\ddiv\bm{u}_{h}\cdot\bm{n}_{f} = \bm{v}_{h}\cdot\bm{n}_{f}\quad\text{for all}~f\in \mathscr{F}.
	\end{equation*}
	Then for any $p_{h}\in P_{k-2}(K;\mathbb{R})/P_1(K;\mathbb{R})$, it holds 
	\begin{equation*}
	\begin{aligned}
	(\ddiv\ddiv \bm{u}_{h}-\ddiv\bm{v}_{h},p_{h})_{K} &= -(\ddiv\bm{u}_{h}-\bm{v}_{h},\nabla p_{h})_{K} \\
	& = -\sum_{f\in\mathscr{F}(K)}(\bm{u}_{h}\bm{n}_{f},\nabla p_{h})_{f} + (\bm{u}_{h},\nabla^2p_{h})_{K} -(\bm{v}_{h},\nabla p_{h})_{K}.
	\end{aligned}
	\end{equation*}
	If $p_{h}\in P_1(K;\mathbb{R})$, then $\nabla p_{h} = (a,b,c)^{T}\in P_0(K;\mathbb{R}^3)$. By induction, there exists $3\#\mathscr{F}$ constants $\cup_{f\in\mathscr{F}}\{a_f,b_f,c_f\}$ which satisfy 
\begin{equation*}
\begin{aligned}
-\sum_{f\in\mathscr{F}(K)}a_{f} &=(\bm{v}_{h},\bm{I}_1)_K,\quad\text{for all}~K\in \mathcal{T}_{h},\\
-\sum_{f\in\mathscr{F}(K)}b_{f}&=(\bm{v}_{h}, \bm{I}_2)_K,\quad\text{for all}~K\in \mathcal{T}_{h},\\
-\sum_{f\in\mathscr{F}(K)}c_{f} &=(\bm{v}_{h}, \bm{I}_3)_K,\quad\text{for all}~K\in \mathcal{T}_{h}.   
\end{aligned}
\end{equation*}	
Here $\bm{I}_1:=(1, 0, 0)^{\intercal}$, $\bm{I}_2:=(0, 1, 0)^{\intercal}$, $\bm{I}_3:=(0, 0, 1)^{\intercal}$, and $\#\mathscr{F}$ denotes the number of faces in the finite set $\mathscr{F}$. 

The second step is to set the DOFs \eqref{DOF:Sig3} by
	\begin{equation*}
    (\bm{u}_{h}\bm{n}_{f}, \bm{I}_1)_f=a_{f},\quad (\bm{u}_{h}\bm{n}_{f}, \bm{I}_2)_f=b_{f},\quad (\bm{u}_{h}\bm{n}_{f}, \bm{I}_3)_f=c_{f},\quad\text{for all}~f\in \mathscr{F},        
	\end{equation*}
	and set the other DOFs to be zero. On each element $K\in\mathcal{T}_\triangle$, the above construction of $\bm{u}_{h}$ shows, for any $p_{h}\in P_{k-2}(K;\mathbb{R})$, that
\begin{equation*}
(\ddiv\ddiv \bm{u}_{h}-\ddiv\bm{v}_{h},p_{h})_{K} =-\sum_{f\in\mathscr{F}(K)}(\bm{u}_{h}\bm{n}_{f},\nabla p_{h})_{f} + (\bm{u}_{h},\nabla^2p_{h})_{K} -(\bm{v}_{h},\nabla p_{h})_{K}=0. 
\end{equation*}
This concludes the proof.
\end{proof}
The exactness of the discrete complex \eqref{eq:DisComp} is proved in the following theorem.
\begin{theorem}
	The finite element sequence 
	\ben
	~~~~~~~~~RT\stackrel{\subseteq}{\longrightarrow}{V}_{k+2,\triangle}\xlongrightarrow[~]{{{\rm{dev}\,  {\ggrad}}}} \mathcal{U}_{k+1,\triangle}\xlongrightarrow[~]{{{\rm{sym}}\,  {\bc}}}{\Sigma}_{k,\triangle}\xlongrightarrow[~]{{ {\rm{div}}\, {\bd}}}{Q}_{k-2,\triangle}{\longrightarrow} 0
	\een
is a divdiv complex, which is exact on a contractible domain.
\end{theorem}
\begin{proof}
	It only remains to check the dimension. Notice that the global dimension of $V_{k+2,\triangle}$ is 
	$$36\#\mathscr{V} + (11k-29)\#\mathscr{E} + (2k^2-11k+15)\#\mathscr{F} + \frac{k^3-4k^2+3k}{2}\#\mathcal{T}_{h}.$$
 Here $\# \mathcal{S}$ is the number of the elements in the finite set $S$. Similarly, by the degrees of freedom defined above, the global dimension of $\mathcal{U}_{k+1, \triangle}$ is 
 $$41\#\mathscr{V} + (16k-30)\#\mathscr{E} + (4k^2-15k+11)\#\mathscr{F} + \frac{4k^3-12k^2-4k+12}{3}\#\mathcal{T}_{h}.$$
The global dimension of $\Sigma_{k, \triangle}$ is 
$$9\#\mathscr{V} + (5k-5)\#\mathscr{E}+ (2k^2-4k)\#\mathscr{F} + (k^3-2k^2-3k)\#\mathcal{T}_{h},$$
 and the global dimension of $Q_{k-2,\triangle}$ is $\frac{k^3-k}{6}\#\mathcal{T}_{h}$.
	
	By Euler's formula $\#\mathscr{V}-\#\mathscr{E}+\#\mathscr{F}-\#\mathcal{T}_{h}=1$, it holds that
	\begin{equation*}
	\operatorname{dim}V_{k+2,\triangle}-\operatorname{dim}\mathcal{U}_{k+1, \triangle}+\operatorname{dim}\Sigma_{k, \triangle}-\operatorname{dim}Q_{k-2,\triangle} = 4=\operatorname{dim}RT.
	\end{equation*}
	This concludes the proof.
\end{proof}

\bibliographystyle{plain}
\bibliography{reference}

\end{document}